\documentclass[a4paper,reqno]{article}
\usepackage{amsmath,amsfonts,amssymb,amsthm,bbm,graphicx,enumerate,geometry}
\usepackage{mathrsfs}

\usepackage{subfigure} %插入多图时用子图显示的宏包

\usepackage{color}

\usepackage{framed} 
\usepackage{hyperref}  
\hypersetup{
    linktoc=page,
    linkcolor=red,          % color of internal links
    citecolor=blue,        % color of links to bibliography
    filecolor=blue,      % color of file links
    urlcolor=cyan,
   colorlinks=true           % color of external links
}
\usepackage[
  backend=biber,
  style=alphabetic,
  autocite=inline,
]{biblatex}

\numberwithin{equation}{section}

\newtheorem{thm}{Theorem}[section]

\newtheorem{defn}[thm]{Definition}
\newtheorem{example}[thm]{Example}
\newtheorem{remark}[thm]{Remark}

\newtheorem{conjecture}[thm]{Conjecture}

\renewcommand{\epsilon}{\varepsilon}

\def\<#1{\langle #1\rangle}
\begin{document}{\allowdisplaybreaks[4]}

%%%%%%%%%%%%%%%%%%%%%%%%%%%%%
%%%%%%%%%%%%%%%%%%%%%%%%%%%%%

\title{Asymptotic of Coulomb gas integral, Temperley-Lieb type algebras and pure partition functions}
\author{
    Jiaxin Zhang
    \footnotemark[1]
   }
\renewcommand{\thefootnote}{\fnsymbol{footnote}}

\footnotetext[1]{{\bf zhangjx@caltech.edu} Department of Mathematics, California Institute of Technology}

\maketitle
\begin{abstract}

In this supplementary note, we study the asymptotic behavior of several types of Coulomb gas integrals following techniques similar to those in \cite{FK15c} and construct the pure partition functions for multiple radial $\mathrm{SLE}(\kappa)$ and general multiple chordal $\mathrm{SLE}(\kappa)$ systems.

For both radial and chordal cases, we prove the linear independence of the ground state solutions $\mathcal{J}_{\alpha}^{(m,n)}(\boldsymbol{x})$ to the null vector equations for irrational values of $\kappa \in (0,8)$.

In particular, we show that the ground state solutions $\mathcal{J}^{(m,n)}_\alpha \in B_{m,n}$, indexed by link patterns $\alpha$ with $m$ screening charges, are linearly independent when $\kappa$ is irrational. This is achieved by constructing, for each link pattern $\beta$, a dual functional $l_\beta \in B^{*}_{m,n}$ such that the meander matrix of the corresponding Temperley-Lieb type algebra is given by $M_{\alpha\beta} = l_{\beta}(\mathcal{J}^{(m,n)}_\alpha)$. The determinant of this matrix admits an explicit expression and is nonzero for irrational $\kappa$, establishing the desired linear independence.

As a consequence, we construct the pure partition functions $\mathcal{Z}_{\alpha}(\boldsymbol{x})$ of the multiple $\mathrm{SLE}(\kappa)$ systems for each link pattern $\alpha$ by multiplying the inverse of the meander matrix.

This method can also be extended to the asymptotic analysis of the excited state solutions $\mathcal{K}_{\alpha}$ in both radial and chordal cases.

\end{abstract}

\newpage 

\tableofcontents

\newpage 
\section{Introduction}

\subsection{Background and main results}

\indent The Schramm-Loewner evolution ($SLE_{\kappa}$) with $\kappa>0$ is a one-parameter family of random conformally invariant curves in the plane, as introduced by Schramm \cite{Sch00}. SLEs are closely related to conformal field theory (CFT), as can be seen in references like \cite{BB03a, Car03, FW03, FK04, Dub15a, Dub15b, Pel19}. The parameter $\kappa$ characterizes the roughness of these fractal curves and determines the central charge $c(\kappa)=(3 \kappa-8)(6-\kappa) / 2 \kappa$ of the associated CFT. 

Standard multiple chordal SLE($\kappa$) systems have been extensively studied in various works including \cite{Dub06,FK15a,KL07,Law09b,PW19,PW20}. Multiple radial SLE($\kappa$) systems have been studied in literature such as \cite{SKFZ11, FKZ12, MZ24,MZ25,Zh25th} and more general multiple chordal SLE($\kappa$) systems in \cite{Z25a,Z25b}.

The Temperley–Lieb-type algebras (see \cite{MDSA13, RSA14}, and section 5 for a brief review) play a fundamental role in the study of two-dimensional critical phenomena, particularly through their connection to loop models, representation theory, and Schramm–Loewner evolutions (SLE). These algebras encode the combinatorial structure of non-crossing pairings and naturally appear in the classification of multiple SLE connectivity patterns via link patterns.

In this note, we compute the asymptotic behavior of Coulomb gas solutions associated with multiple SLE($\kappa$) systems.

Our main observation is that these asymptotics coincide with the entries of the meander matrix, which is the Gram matrix for the standard module of the Temperley–Lieb algebra.

Let $\mathcal{J}^{(m,n)}_\alpha$ (see definitions \ref{chordal ground} and \ref{radial ground}) denote the Coulomb gas integral corresponding to a link pattern $\alpha$ constructed using screening charges along prescribed contours. Each such function admits a well-defined asymptotic when the insertion points on the real line are fused according to a link pattern $\beta$.

\begin{thm}[Chordal asymptotics and Gram pairing]\label{chordal meander}
For a fixed link pattern $\beta$, define the evaluation functional $l_\beta$ acting on Coulomb gas integrals by taking the corresponding asymptotic limit:
\[
l_{\beta}\left( \mathcal{J}^{(m,n)}_\alpha \right) := \lim_{\text{asymptotic with respect to } \beta} \mathcal{J}^{(m,n)}_\alpha.
\]
Then we have the identity:
\[
l_{\beta} \left( \mathcal{J}^{(m,n)}_\alpha \right) = \left\langle \alpha , \beta \right\rangle_{n,m},
\]
where $\delta= n(\kappa) = -2\cos\left(\frac{4\pi}{\kappa}\right)$, the right-hand side is the inner product in the standard module $V_{n,m}$ computed via the Gram (meander) matrix.
\end{thm}

\begin{defn}[Inner product on standard module]
Let $x, y \in V_{n,m}$ be two $(n,m)$-link states and $\delta \in \mathbb{C}$. Define $\langle x, y \rangle_{n,m}$ as follows:
\begin{enumerate}
    \item Reflect $x$ across a vertical axis (mirror image).
    \item Glue $x$ to $y$ vertically, forming a closed link diagram.
    \item If all defects pair, and $l(\alpha,\beta)$ closed loops are formed, then
    \[
    \langle x, y \rangle_{n,m} := \delta^{l(\alpha,\beta)}.
    \]
    \item Otherwise, if any defects remain unpaired, then $\langle x, y \rangle_{n,m} := 0$.
\end{enumerate}
\end{defn}

\begin{defn}[Chordal meander Matrix]
Let $LP(n,m)$ be the set of chordal link patterns on $n$ boundary points, consisting of $m$ non-crossing pairings of $n$ points on the real line. The \emph{meander matrix} $\mathcal{M}_{\kappa}(\alpha,\beta)$ is the Gram matrix associated with the bilinear form
\[ \langle \alpha , \beta \rangle_{n,m}
\]
where $\alpha, \beta \in LP(n,m)$, $\delta = n(\kappa)=-2\cos(\frac{4\pi}{\kappa})$ is a loop weight parameter.

\end{defn}

\begin{thm}[Radial asymptotics and affine Gram pairing]\label{radial meander}
Let $\mathcal{J}^{(m,n)}_\alpha$ denote the Coulomb gas integral associated with an affine link pattern $\alpha$, and let $\beta$ be another such pattern. Define the evaluation functional $l_{\beta}$ by:
\[
l_{\beta}\left( \mathcal{J}^{(m,n)}_\alpha \right) := \lim_{\text{asymptotic with respect to } \beta} \mathcal{J}^{(m,n)}_\alpha.
\]
Then we have the identity:
\[
l_{\beta} \left( \mathcal{J}^{(m,n)}_\alpha \right) = \left\langle \alpha \mid \beta \right\rangle_{n,m},
\]
where the Gram pairing is computed with loop weights $a = n(\kappa) = -2\cos(\frac{4\pi}{\kappa})$,  $b = 2$.
\end{thm}
\begin{defn}[Inner product on affine standard module]
Let $x, y \in V_{n,m}$ be two radial $(n,m)$-link states and $a,b \in \mathbb{C}$. Define the inner product $\langle x \mid y \rangle_{n,m}$ as follows:
\begin{enumerate}
    \item Reflect $x$ across a vertical axis (mirror image).
    \item Glue $x$ to $y$ vertically, forming a closed link diagram called $\mathcal{G}(\alpha, \beta)$.
    \item If all defects pair, and $m$ closed loops are formed, then
    \[
    \langle x \mid y \rangle_{n,m} := a^{n_a}b^{n_b}.
    \]
    where $n_a$ and $n_b$ denote the numbers of contractible and non-contractible loops formed in the diagram $\mathcal{G}(\alpha, \beta)$, and $a, b$ are loop weights.
    \item Otherwise, if any defects remain unpaired, then $\langle x \mid y \rangle_{n,m} := 0$.
\end{enumerate}
\end{defn}

\begin{defn}[Radial meander Matrix]
Let $LP(n,m)$ be the set of chordal link patterns on $n$ boundary points, consisting of $m$ non-crossing pairings of $n$ points on the real line. The \emph{meander matrix} $\mathcal{M}_{\kappa}(\alpha,\beta)$ is the Gram matrix associated with the bilinear form
\[ \langle \alpha \mid \beta \rangle_{n,m}
\]
where $\alpha, \beta \in LP(n,m)$, $a = n(\kappa)=-2\cos(\frac{4\pi}{\kappa})$ and $b=2$ are two loop weight parameters.

\end{defn}

\begin{proof}[Sketch of proof of theorem \ref{chordal meander} and \ref{radial meander}]
This follows from the asymptotics of the Coulomb gas integrals: in the limit where insertion points coalesce according to the arcs in $\beta$, the contour integrals simplify. The resulting expression matches the diagrammatic pairing rules of the Gram form.
\end{proof}

\begin{thm}{See \cite{RSA14,MDSA13}}
    The meander matrices $\mathcal{M}_{\alpha\beta}$ in both radial and chordal cases are invertible for $\kappa \notin \mathbb{Q}$.
\end{thm}

The pure partition function $\mathcal{Z}_\beta$, formally defined in (\ref{pure partition function}), corresponds to a multiple SLE($\kappa$) system exhibiting a prescribed link pattern $\beta$.
As an application of the main theorem, we construct a family of such partition functions in Section~\ref{asymptotics and connectivity} via the formula
\begin{equation}
\mathcal{Z}_\beta(\boldsymbol{x}) = \sum_{\alpha \in \operatorname{LP}(n,m)} \mathcal{M}_\kappa(\alpha, \beta)^{-1} \mathcal{J}_\alpha^{(m,n)}(\boldsymbol{x}), \quad \beta \in \mathrm{LP}(n,m).
\end{equation}

The computational framework developed here also extends to the asymptotic analysis of excited solutions, as discussed in Section~\ref{asymptotic of exicited solutions}.

\section{Coulomb gas solutions integral to null vector equations}
\subsection{Chordal ground and excited solutions}
\label{Classification of screening solutions}
\
\indent
Based on the Coulomb gas integral method introduced in \cite{Z25a, Z25b}, we construct chordal ground and excited solutions that satisfy the null vector equations (also known as BPZ equations), which arise from the insertion of degenerate fields in conformal field theory. These solutions also obey three independent Ward identities that reflect conformal covariance under Möbius transformations.

\begin{equation} \label{null vector equation for Screening solutions}
\left[\frac{\kappa}{4} \partial_j^2+\sum_{k \neq j}^{n}\left(\frac{\partial_k}{x_k-x_j}-\frac{(6-\kappa) /  2\kappa}{\left(x_k-x_j\right)^2}\right)+\frac{\partial_{n+1}}{u-x_j} \right.
\left. -\frac{\lambda_{(b)}(u)}{\left(u-x_j\right)^2}\right] \mathcal{J}\left(\boldsymbol{x},u\right)=0
\end{equation}
for $j=1,2,\ldots,n$,
and the following ward identities: 
\begin{equation} \label{Ward identities for screening solutions}
\begin{aligned}
&\left[\sum_{i=1}^{n} \partial_{x_i}+ \partial_u\right] \mathcal{J}(\boldsymbol{x}) =0,\\
& \left[\sum_{i=1}^{n}\left(x_i \partial_{x_i}+\frac{6-\kappa}{2\kappa}\right)+ u\partial_u+\lambda_{(b)}(u)u\right]\mathcal{J}(\boldsymbol{x})=0, \\
& \left[\sum_{i=1}^{n}\left(x_i^2 \partial_{x_i}+\frac{6-\kappa}{\kappa}x_i\right)+ u^2\partial_u+2\lambda_{(b)}(u)u\right] \mathcal{J}(\boldsymbol{x})=0
\end{aligned}
\end{equation}
where $\lambda_{(b)}(u)$ is the conformal dimensions of $u$.

We need to choose a set of integration contours to integrate $\Phi$. we will explain how we choose integration contours, which lead the screening solutions. We conjecture that these screening solutions span the solution space of the null vector equations (\ref{null vector equation for Screening solutions}) 
and the Ward's identities (\ref{Ward identities for screening solutions}).

To do this, let's begin by defining the link patterns that characterize the topology of integration contours.

\begin{defn}[Chordal link pattern] Given $\boldsymbol{x}=\{x_1,x_2,...,x_n\}$ on the real line, a link pattern is a homotopically equivalent class of non-intersecting curves connecting pair of boundary points (links) or connecting boundary points and the infinity (rays). The link patterns with $n$ boundary points and $m$ links are called $(n,m)$-links, denoted by ${\rm LP}(n,m)$.

The number of chordal $(n,m)$-links is given by $|{\rm LP}(n,m)|=C_{n}^{m+1}-C_{n}^{m}$.
\end{defn}

When all $\sigma_i= a$, $1\leq i \leq n$, we can assign charge $-2a$ or $2(a+b)$ to screening charges.

\begin{itemize}

\item (Chordal ground solutions) \label{chordal ground}
In the upper half plane $\mathbb{H}$, we assign charge $a$ to $x_1,x_2,\ldots,x_n$, charge $-2a$ to $\xi_1,\ldots,\xi_m$ and charge $\sigma_u=2b-(n-2m)a$ to marked points $u$ 
to maintain neutrality condition ($\rm{NC_b}$).
\begin{equation} \label{multiple chordal SLE kappa master function}
\begin{aligned}
\Phi_{\kappa}\left(x_1, \ldots, x_{n}, \xi_1,\xi_2,\ldots,\xi_m, u \right)= & \prod_{i<j}\left(x_i-x_j\right)^{a^2} \prod_{j<k}\left(x_j-\xi_k\right)^{-2a^2}  \prod_{j<k}\left(\xi_j-\xi_k\right)^{4a^2}  \\
& \prod_{j}(x_i-u)^{a(2b-(n-2m)a)}
\prod_{j}(\xi_j-u)^{-2a(2b-(n-2m)a)}
\end{aligned}
\end{equation}

\begin{itemize}
    \item[(1)] $(-2a)\cdot a=-\frac{4}{\kappa}$.  $\xi_i=x_j$ is a singular point of the type $\left(\xi_i-x_j\right)^{-4 / \kappa}$.
     \item[(2)] $(-2a)\cdot (-2a)=\frac{8}{\kappa}$. $\xi_i=\xi_j$ is a singular point of of the type $(\xi_i-\xi_j)^{\frac{8}{\kappa}}$
     \item[(3)] $(-2a)\cdot (2b-(n-2m)a)=\frac{4(n-2m+2)}{\kappa}$.  $\xi_i=u$ is singular point of the type $(\xi_i-u)^{\frac{4(n-2m+2)}{\kappa}}$.
\end{itemize}

In this case, for $m \leq \frac{n+2}{2}$ and a $(n,m)$ chordal link pattern $\alpha$, we can choose $p$ non-intersecting Pochhammer contours $\mathcal{C}_1,\mathcal{C}_2,\ldots,\mathcal{C}_m$ surrounding pairs of points (which correspond to links in a chordal link pattern) to integrate $\Phi_{\kappa}$, we obtain
\begin{equation}
    \mathcal{J}^{(m, n)}_{\alpha}(\boldsymbol{x}):=\oint_{\mathcal{C}_1} \ldots \oint_{\mathcal{C}_m} \Phi_\kappa(\boldsymbol{x}, \boldsymbol{\xi}) d \xi_m \ldots d \xi_1 .
\end{equation}

Note that the charge at $u$ is given by $\sigma_u=2b-(n-2m)a$, thus $$\lambda_{(b)}(u)= \frac{(2m-n)^2}{\kappa}-\frac{2(2m-n)}{\kappa}+\frac{2m-n}{2}$$

The chordal ground solution $\mathcal{J}^{(m,n)}_{\alpha}$ satisfies the null vector equations (\ref{null vector equation for Screening solutions}) and Ward's identities (\ref{Ward identities for screening solutions}) with above $\lambda_{(b)}(u)$.

The number of solutions we can construct via screening is precisely the number of chordal link patterns. We conjecture that these are exactly all the solutions to the null vector equations.

\item (Chordal excited solutions)In the upper half plane $\mathbb{H}$, we assign charge $a$ to $x_1,x_2,\ldots,x_n$, charge $-2a$ to $\xi_1,\ldots,\xi_m$ and charge $2(a+b)$ to $\xi_1,\ldots,\zeta_q$.
Then, we assign charge $\sigma_u=2b-(n-2m)a-2q(a+b)$ to marked points $u$ 
to maintain neutrality condition ($\rm{NC_b}$).
\begin{equation}
\begin{aligned}
&\Phi_{\kappa}\left(x_1, \ldots, x_{n}, \xi_1,\xi_2,\ldots,\xi_m,\zeta_1,\zeta_2,\ldots,\zeta_q, u \right)=  \\
&\prod_{i<j}\left(x_i-x_j\right)^{a^2} \prod_{j<k}\left(x_j-\xi_k\right)^{-2a^2}  \prod_{j<k}\left(\xi_j-\xi_k\right)^{4a^2}  \\
& \prod_{j<k}\left(x_j-\zeta_k\right)^{2a(a+b)}  \prod_{j<k}\left(\zeta_j-\zeta_k\right)^{4(a+b)^2} 
\\
& \prod_{j}(x_i-u)^{a\sigma_u}
\prod_{j}(\xi_j-u)^{-2a\sigma_u}
\prod_{j}(\zeta_j-u)^{2(a+b)\sigma_u}
\end{aligned}
\end{equation}

In the unit disk $\mathbb{H}$, if we set $u=\infty$, then we have

\begin{equation}
\begin{aligned}
&\Phi_{\kappa}\left(x_1, \ldots, x_{n}, \xi_1,\xi_2,\ldots,\xi_m,\zeta_1,\zeta_2,\ldots,\zeta_q\right)=  \\
&\prod_{i<j}\left(x_i-x_j\right)^{a^2} \prod_{j<k}\left(x_j-\xi_k\right)^{-2a^2}  \prod_{j<k}\left(\xi_j-\xi_k\right)^{4a^2}  \\
& \prod_{j<k}\left(x_j-\zeta_k\right)^{2a(a+b)}  \prod_{j<k}\left(\zeta_j-\zeta_k\right)^{4(a+b)^2} 
\end{aligned}
\end{equation}

\begin{itemize}
    \item[(1)] $(-2a)\cdot a=-\frac{4}{\kappa}$.  $\xi_i=z_j$ is a singular point of the type $\left(\xi_i-z_j\right)^{-4 / \kappa}$.
     \item[(2)] $(-2a)\cdot (-2a)=\frac{8}{\kappa}$. $\xi_i=\xi_j$ is a singular point of of the type $(\xi_i-\xi_j)^{\frac{8}{\kappa}}$
     \item[(3)] $(-2a)\cdot (b-\frac{(n-2m)a}{2}-q(a+b))=\frac{2(n-2m+2)}{\kappa}+q$.  $\xi=u$ and $\xi=u^*$ are singular points of the type $(\xi_i-u)^{\frac{2(n-2m+2)}{\kappa}+q}$ and $(\xi_i-u^*)^{\frac{2(n-2m+2)}{\kappa}+q}$
     \item[(4)] $2(a+b)\cdot (2b-(n-2m)a-2q(a+b))=\frac{(1-q)\kappa}{2}-n+2m-2$.  $\xi_i=u$ is a singular point of the type $(\xi_i-u)^{\frac{(1-q)\kappa}{2}-n+2m-2}$.
\end{itemize}

For $q=1$, $\zeta_1=u$ is one singular point of degree $-n+2m-2$.
We have only one choice for screening contours to integrate $\zeta_1$, the circle $C(u,\epsilon)$ around $u$ with radius $\epsilon$, this gives the excited solution.
    
In this case, for $m \leq \frac{n+2}{2}$ and a $(n,m)$ chordal link pattern $\alpha$, we can choose $p$ non-intersecting Pochhammer contours $\mathcal{C}_1,\mathcal{C}_2,\ldots,\mathcal{C}_m$ surrounding pairs of points (which correspond to links in a chordal link pattern) to integrate $\Phi_{\kappa}$, we obtain
\begin{equation}
    \mathcal{K}^{(m,n)}_{\alpha}(\boldsymbol{z}):=\oint_{\mathcal{C}_1} \ldots \oint_{\mathcal{C}_m}\oint_{C(u,\epsilon)} \Phi_\kappa(\boldsymbol{z}, \boldsymbol{\xi}) d \xi_m \ldots d \xi_1 d\zeta_1.
\end{equation}
\end{itemize} 

Note that the charges at $u$ is given by $\sigma_u=(2m-n-2)a$
$$\lambda_{(b)}(u)= \frac{(2m-n)(2m-n-2)}{\kappa}-\frac{2m-n-2}{2}$$ 
The chordal excite solution $\mathcal{K}^{(m,n)}_{\alpha}$ 
satisfies the null vector equations (\ref{null vector equation for Screening solutions}) and Ward's identities (\ref{Ward identities for screening solutions}) with above $\lambda_{(b)}(u)$.

For $q\geq 2$, since $u$ is the only singular points for screening charges,
it is impossible to choose two non-intersecting contours for $\{\zeta_1,\zeta_2,\ldots,\zeta_q \}$.

We conjecture that the set of screening solutions constructed via Coulomb gas integrals indexed by chordal link patterns forms a basis for the space of solutions to the null vector equations and Ward identities in the presence of a marked boundary point.

\subsection{Radial ground and excited solutions}
\
\indent
In \cite{MZ24, Zh25th}, we construct both radial ground and excited solutions, denoted respectively by $\mathcal{J}_\alpha(\boldsymbol{x})$ and $\mathcal{K}_\alpha(\boldsymbol{x})$, using the Coulomb gas formalism.

\begin{equation} 
\left[\frac{\kappa}{4} \partial_j^2+\sum_{k \neq j}^{n}\left(\frac{\partial_k}{z_k-z_j}-\frac{(6-\kappa) /  2\kappa}{\left(z_k-z_j\right)^2}\right)+\frac{\partial_{n+1}}{u-z_j} +\frac{\partial_{n+2}}{u^*-z_j} \right.
\left. -\frac{\lambda_{(b)}(u)}{\left(u-z_j\right)^2}-\frac{\lambda_{(b)}(u^*)}{\left(u^*-z_j\right)^2} \right] \mathcal{J}\left(\boldsymbol{z},u\right)=0
\end{equation}
for $j=1,2,\ldots,n$,
and the following ward identities : 
\begin{equation} 
\begin{aligned}
&\left[\sum_{i=1}^{n} \partial_{x_i}+ \partial_z+\partial_{z^*}\right] \mathcal{J}(\boldsymbol{x},z) =0,\\
& \left[\sum_{i=1}^{n}\left(x_i \partial_{x_i}+\frac{6-\kappa}{2\kappa}\right)+ z\partial_z+\lambda_{(b)}(z)z+z^*\partial_{z^*}+\lambda_{(b)}(z^*)z^*\right]\mathcal{J}(\boldsymbol{x},z)=0, \\
& \left[\sum_{i=1}^{n}\left(x_i^2 \partial_{x_i}+\frac{6-\kappa}{\kappa}x_i\right)+ z^2\partial_z+2\lambda_{(b)}(z)z+(z^*)^2\partial_{z^*}+2\lambda_{(b)}(z^*)z^*\right] \mathcal{J}(\boldsymbol{x},z)=0
\end{aligned}
\end{equation}
where $\lambda_{(b)}(z)$ and $\lambda_{(b)}(z^*)$ are the conformal dimensions of $z$ and $z^*$.
\begin{itemize}

\item (Radial ground solutions) \label{radial ground}
In the upper half plane $\mathbb{H}$, we assign charge $a$ to $x_1,x_2,\ldots,x_n$, charge $-2a$ to $\xi_1,\ldots,\xi_m$ and charge $\sigma_z=\sigma_{z^*}=b-\frac{(n-2m)a}{2}$ to marked points $z$ and $z^*$ 
to maintain neutrality condition ($\rm{NC_b}$).
\begin{equation} \label{multiple radial SLE(kappa) master function in H}
\begin{aligned}
\Phi_{\kappa}\left(x_1, \ldots, x_{n}, \xi_1,\xi_2,\ldots,\xi_m, z \right)= & \prod_{i<j}\left(x_i-x_j\right)^{a^2} \prod_{j<k}\left(x_j-\xi_k\right)^{-2a^2}  \prod_{j<k}\left(\xi_j-\xi_k\right)^{4a^2}  \\
& \prod_{j}(x_i-z)^{a(b-\frac{(n-2m)a}{2})}
\prod_{j}(x_i-z^*)^{a(b-\frac{(n-2m)a}{2})}
\\
&\prod_{j}(\xi_j-z)^{-2a(b-\frac{(n-2m)a}{2})}\prod_{j}(\xi_j-z^*)^{-2a(b-\frac{(n-2m)a}{2})}
\end{aligned}
\end{equation}

\begin{itemize}
    \item[(1)] $(-2a)\cdot a=-\frac{4}{\kappa}$.  $\xi_i=x_j$ is a singular point of the type $\left(\xi_i-x_j\right)^{-4 / \kappa}$.
     \item[(2)] $(-2a)\cdot (-2a)=\frac{8}{\kappa}$. $\xi_i=\xi_j$ is a singular point of of the type $(\xi_i-\xi_j)^{\frac{8}{\kappa}}$
     \item[(3)] $(-2a)\cdot (b-\frac{(n-2m)a}{2})=\frac{2(n-2m+2)}{\kappa}$.  $\xi=z$ and $\xi=z^*$ are singular points of the type $(\xi_i-z)^{\frac{2(n-2m+2)}{\kappa}}$ and $(\xi_i-z^*)^{\frac{2(n-2m+2)}{\kappa}}$
\end{itemize}

In this case, for $m \leq \frac{n+2}{2}$ and a $(n,m)$ radial link pattern $\alpha$, we can choose $p$ non-intersecting Pochhammer contours $\mathcal{C}_1,\mathcal{C}_2,\ldots,\mathcal{C}_m$ surrounding pairs of points (which correspond to links in a radial link pattern) to integrate $\Phi_{\kappa}$, we obtain
\begin{equation}
    \mathcal{J}^{(m, n)}_{\alpha}(\boldsymbol{x}):=\oint_{\mathcal{C}_1} \ldots \oint_{\mathcal{C}_m} \Phi_\kappa(\boldsymbol{x}, \boldsymbol{\xi}) d \xi_m \ldots d \xi_1 .
\end{equation}
In particular, if $m=0$, we call $\Phi_{\kappa}$ the fermionic ground solution.

Note that the charges at $u$ and $u^*$ are given by $\sigma_u=\sigma_{u^*}=b-\frac{(n-2m)a}{2}$, thus $$\lambda_{(b)}(u)=\lambda_{(b)}(u^*)=\frac{(n-2m)^2a^2}{8}-\frac{b^2}{2}= \frac{(n-2m)^2}{4\kappa}-\frac{(\kappa-4)^2}{16\kappa}$$

The radial ground solution $\mathcal{J}^{(m,n)}_{\alpha}$ satisfies the null vector equations (\ref{null vector equation for Screening solutions}) and Ward's identities (\ref{Ward identities for screening solutions}) with above $\lambda_{(b)}(u)$ and $\lambda_{(b)}(u^*)$

\item (Radial excited solutions)In the upper half plane $\mathbb{H}$, we assign charge $a$ to $x_1,x_2,\ldots,x_n$, charge $-2a$ to $\xi_1,\ldots,\xi_m$ and charge $2(a+b)$ to $\zeta_1,\ldots,\zeta_q$.
Then, we assign charge $\sigma_u=\sigma_{u^*}=b-\frac{(n-2m)a+2q(a+b)}{2}$ to marked points $u$ and $u^*$ 
to maintain neutrality condition ($\rm{NC_b}$).
\begin{equation}
\begin{aligned}
&\Phi_{\kappa}\left(x_1, \ldots, x_{n}, \xi_1,\xi_2,\ldots,\xi_m,\zeta_1,\zeta_2,\ldots,\zeta_q, u \right)=  \\
&\prod_{i<j}\left(x_i-x_j\right)^{a^2} \prod_{j<k}\left(x_j-\xi_k\right)^{-2a^2}  \prod_{j<k}\left(\xi_j-\xi_k\right)^{4a^2}  \\
& \prod_{j<k}\left(x_j-\zeta_k\right)^{2a(a+b)}  \prod_{j<k}\left(\zeta_j-\zeta_k\right)^{4(a+b)^2} 
\\
& \prod_{j}(x_i-u)^{a\sigma_u}
\prod_{j}(x_i-u^*)^{a\sigma_{u^*}}
\\
&\prod_{j}(\xi_j-u)^{-2a\sigma_u}\prod_{j}(\xi_j-u^*)^{-2a\sigma_{u^*}}
\\
&\prod_{j}(\zeta_j-u)^{2(a+b)\sigma_u}\prod_{j}(\zeta_j-u^*)^{2(a+b)\sigma_{u^*}}
\end{aligned}
\end{equation}

\begin{itemize}
    \item[(1)] $(-2a)\cdot a=-\frac{4}{\kappa}$.  $\xi_i=x_j$ is a singular point of the type $\left(\xi_i-x_j\right)^{-4 / \kappa}$.
     \item[(2)] $(-2a)\cdot (-2a)=\frac{8}{\kappa}$. $\xi_i=\xi_j$ is a singular point of of the type $(\xi_i-\xi_j)^{\frac{8}{\kappa}}$
     \item[(3)] $(-2a)\cdot (b-\frac{(n-2m)a}{2}-q(a+b))=\frac{2(n-2m+2)}{\kappa}+q$.  $\xi=u$ and $\xi=u^*$ are singular points of the type $(\xi_i-u)^{\frac{2(n-2m+2)}{\kappa}+q}$ and $(\xi_i-u^*)^{\frac{2(n-2m+2)}{\kappa}+q}$
     \item[(4)] $2(a+b)\cdot (b-\frac{(n-2m)a}{2}-q(a+b))=\frac{(1-q)\kappa}{4}+\frac{-n+2m-2}{2}$.  $\xi=u$ and $\xi=u^*$ are singular points of the type $(\xi_i-u)^{\frac{(1-q)\kappa}{4}+\frac{-n+2m-2}{2}}$ and $(\xi_i-u^*)^{\frac{(1-q)\kappa}{4}+\frac{-n+2m-2}{2}}$
\end{itemize}

For $q=1$, $\zeta_1=u$ and $\zeta_1=u^*$ are two singular points of degree $\frac{-n+2m-2}{2}$.
We have two choices for screening contours to integrate $\zeta_1$
\begin{itemize}
    \item $n$ odd, Pochhammer contour $\mathscr{P}(u,u^*)$ surrounding $u$ and $u^*$, however,
    $$\int_{\mathscr{P}(u,u^*)}\Phi_{\kappa} d\zeta = 0$$
    
    \item $n$ even, the circle $C(u,\epsilon)$ around $u$ with radius $\epsilon$, this gives the excited solution
    
In this case, for $m \leq \frac{n+2}{2}$ and a $(n,m)$ radial link pattern $\alpha$, we can choose $p$ non-intersecting Pochhammer contours $\mathcal{C}_1,\mathcal{C}_2,\ldots,\mathcal{C}_m$ surrounding pairs of points (which correspond to links in a radial link pattern) to integrate $\Phi_{\kappa}$, we obtain
\begin{equation}
    \mathcal{K}^{(m,n)}_{\alpha}(\boldsymbol{x}):=\oint_{\mathcal{C}_1} \ldots \oint_{\mathcal{C}_m}\oint_{C(u,\epsilon)} \Phi_\kappa(\boldsymbol{x}, \boldsymbol{\xi}) d \xi_m \ldots d \xi_1 d\zeta_1.
\end{equation}
In particular, if $p=0$, we call $\Phi_{\kappa}$ the fermionic excited solution.
\end{itemize} 

Note that the charges at $z$ and $z^*$ are given by $\sigma_z=\sigma_{z^*}=\frac{(2m-n-2)a}{2}$
$$\lambda_{(b)}(z)=\lambda_{(b)}(z^*)=\frac{(n-2m+\frac{\kappa}{2})^2}{4\kappa}-\frac{(\kappa-4)^2}{16\kappa}$$ 
The radial excited solution $\mathcal{K}^{(m,n)}_{\alpha}$ 
satisfies the null vector equations (\ref{null vector equation for Screening solutions}) and Ward's identities (\ref{Ward identities for screening solutions}) with above $\lambda_{(b)}(u)$ and $\lambda_{(b)}(u^*)$

For $q\geq 2$, since $u$ and $u^*$ are the only singular points for screening charges,
it is impossible to choose two non-intersecting contours for $\{\zeta_1,\zeta_2,\ldots,\zeta_q \}$.

\end{itemize}

\section{Asymptotics of chordal partition functions as interval collapse}

In this section, we study the asymptotic behavior of the Coulomb gas integral $\mathcal{J}^{(m,n)}_\alpha(\boldsymbol{x},u)$
\begin{equation}
\begin{aligned}
\mathcal{J}^{(m,n)}_\alpha(\boldsymbol{x},u) =&\left[\frac{n(\kappa) \Gamma(2-8 / \kappa)}{4 \sin ^2(4 \pi / \kappa) \Gamma(1-4 / \kappa)^2}\right]^m \\
&\left(\prod_{i<j}^{n}\left(x_j-x_i\right)^{2 / \kappa}\right) \left(\prod_{i=1}^{n}\left|u-x_i\right|^{(\kappa-2n+4m-4) / \kappa}\right) \\
&\oint_{\Gamma_1} \ldots \oint_{\Gamma_{m}} d u_1 \ldots d u_{m}\left(\prod_{k=1}^{n} \prod_{l=1}^{m}\left(x_k-u_l\right)^{-4 / \kappa}\right)\left(\prod_{p<q}^{m}\left(u_p-u_q\right)^{8 / \kappa}\right) \\
&
\times\left(\prod_{k=1}^{m}\left(u-u_k\right)^{2(2n-4m+4-\kappa) / \kappa}\right)
\end{aligned}
\end{equation}

When $u$ is fixed at $\infty$, we can simplify the formula to:
\begin{equation}
\begin{aligned}
\mathcal{J}^{(m,n)}_\alpha(\boldsymbol{x}) =&\left[\frac{n(\kappa) \Gamma(2-8 / \kappa)}{4 \sin ^2(4 \pi / \kappa) \Gamma(1-4 / \kappa)^2}\right]^m \left(\prod_{i<j}^{n}\left(x_j-x_i\right)^{2 / \kappa}\right) \\
& \oint_{\Gamma_1} \ldots \oint_{\Gamma_{m}} d u_1 \ldots d u_{m}\left(\prod_{k=1}^{n} \prod_{l=1}^{m}\left(x_k-u_l\right)^{-4 / \kappa}\right)\left(\prod_{p<q}^{m}\left(u_p-u_q\right)^{8 / \kappa}\right) 
\end{aligned}
\end{equation}
\begin{thm}\label{Asymptotic analysis of Coulomb gas correlation}
\indent
\begin{itemize}
 \item[Configuration 1] Neither $x_i$ nor $x_{i+1}$ are endpoints of an integration contour of $\mathcal{J}^{(m,n)}_\alpha(\boldsymbol{x})$.

\begin{equation}
   \lim_{x_i,x_{i+1}\rightarrow p} (x_{i+1}-x_i)^{\frac{6}{\kappa}-1}\mathcal{J}^{(m,n)}_\alpha(\boldsymbol{x}) = 0   
\end{equation}

\item[Configuration 2] Both $x_i$ and $x_{i+1}$ are endpoints of a single, common integration contour $\Gamma_1$ of $\alpha$. Hence, this contour is $\Gamma_1=\mathscr{P}\left(x_i, x_{i+1}\right)$ or $\left[x_i, x_{i+1}\right]^{+}$. (The superscript + indicates that we form the contour $\left[x_i, x_{i+1}\right]^{+}$by slightly bending $\left[x_i, x_{i+1}\right]$ into the upper half-plane, keeping the endpoints fixed.)

\begin{equation}
   \lim_{x_i,x_{i+1}\rightarrow p} (x_{i+1}-x_i)^{\frac{6}{\kappa}-1}\mathcal{J}^{(m,n)}_\alpha(\boldsymbol{x}) = n(\kappa) \mathcal{J}^{(m,n)}_{\hat{\alpha}} (\boldsymbol{\hat{x}} )
\end{equation}

\item[Configuration 3] Either $x_i$ or $x_{i+1}$ is an endpoint of a single integration contour $\Gamma_1$ of $\mathcal{J}^{(m,n)}_\alpha(\boldsymbol{x})$, but the other is not an endpoint of any contour. 

We assume that $i>1, \kappa>4, x_i$ is an endpoint of $\Gamma_1$, and $\Gamma_1$ does not pass over the interval $\left(x_i, x_{i+1}\right)$.

With $\kappa>4$, $\Gamma_1$ is a simple contour, and we decompose it into one simple contour $\Gamma_1^{\prime}$ with its right endpoint at $x_{i-1}$ and another $\Gamma_1^{\prime \prime}$ with its endpoints at $x_i$ and $x_{i-1}$. (We might have $\Gamma_1^{\prime}=\emptyset$ and $\Gamma_1=\Gamma_1^{\prime \prime}$.) Also, the limit breaks into
\begin{equation} 
\begin{aligned}
\lim _{x_{i+1} \rightarrow x_i}\left(x_{i+1}-x_i\right)^{6 / \kappa-1} \mathcal{J}^{(m,n)}_\alpha= & \lim _{x_{i+1} \rightarrow x_i}\left(x_{i+1}-x_i\right)^{6 / \kappa-1}\left(\left.\mathcal{J}^{(m,n)}_\alpha\right|_{\Gamma_1 \mapsto \Gamma_1^{\prime}}\right) \\
& +\lim _{x_{i+1} \rightarrow x_i}\left(x_{i+1}-x_i\right)^{6 / \kappa-1}\left(\left.\mathcal{J}^{(m,n)}_\alpha\right|_{\Gamma_1 \mapsto \Gamma_1^{\prime \prime}}\right) .
\end{aligned}
\end{equation}
The first limit on the right side of \ref{config3} falls under case 1 and, therefore, vanishes. Meanwhile, the second limit on the right side of \ref{config3}  still falls under case 3. 

\begin{equation}
   \lim_{x_i,x_{i+1}\rightarrow p} (x_{i+1}-x_i)^{\frac{6}{\kappa}-1}\mathcal{J}^{(m,n)}_\alpha(\boldsymbol{x}) = \mathcal{J}^{(m-1,n)}_{\hat{\alpha}} (\boldsymbol{\hat{x}} )
\end{equation}
   
\item[Configuration 4]$x_i$ and $x_{x+1}$ are endpoints of one contour $\Gamma_1$,  we collapse $x_i$ and $x_{i+1}$ in the complement arc.
\begin{equation} \label{config4}
\begin{aligned}
\lim _{x_{i+1} \rightarrow x_i}\left(x_{i+1}-x_i\right)^{6 / \kappa-1} \mathcal{J}^{(m,n)}_\alpha= & \lim _{x_{i+1} \rightarrow x_i}\left(x_{i+1}-x_i\right)^{6 / \kappa-1}\left(\left.\mathcal{J}^{(m,n)}_\alpha\right|_{\Gamma_1 \mapsto \Gamma_1^{\prime}}\right) \\
& +\lim _{x_{i+1} \rightarrow x_i}\left(x_{i+1}-x_i\right)^{6 / \kappa-1}\left(\left.\mathcal{J}^{(m,n)}_\alpha\right|_{\Gamma_1 \mapsto \Gamma_1^{\prime \prime}}\right) 
\\
& +\lim _{x_{i+1} \rightarrow x_i}\left(x_{i+1}-x_i\right)^{6 / \kappa-1}\left(\left.\mathcal{J}^{(m,n)}_\alpha\right|_{\Gamma_1 \mapsto \Gamma_1^{\prime \prime\prime}}\right) .
\end{aligned}
\end{equation}

\begin{equation}
   \lim_{x_i,x_{i+1}\rightarrow p} (x_{i+1}-x_i)^{\frac{6}{\kappa}-1}\mathcal{J}^{(m,n)}_\alpha(\boldsymbol{x}) =2 \mathcal{J}^{(m-1,n)}_{\hat{\alpha}} (\boldsymbol{\hat{x}} )
\end{equation}

\item[Configuration 5]  $x_i$ is an endpoint of one contour $\Gamma_1$ , and $x_{i+1}$ is an endpoint of a different contour $\Gamma_2$.

Calculation: With $\kappa>4$ and $\Gamma_1$ and $\Gamma_2$ simple contours, we decompose $\Gamma_1$ (resp. $\Gamma_2$ ) into one simple contour $\Gamma_1^{\prime}$ (resp. $\Gamma_2^{\prime}$ ) with an endpoint at $x_{i-1}$ (resp. $x_{i+2}$ ) and another $\Gamma_1^{\prime \prime}$ (resp. $\Gamma_2^{\prime \prime}$ ) with its endpoints at $x_{i-1}$ and $x_i$ (resp. $x_{i+1}$ and $x_{i+2}$ ). In some cases, we might have $\Gamma_1^{\prime}=\emptyset$ and $\Gamma_1^{\prime \prime}=\Gamma_1$ (resp. $\Gamma_2^{\prime}=\emptyset$ and $\Gamma_2^{\prime \prime}=\Gamma_2$ ). This decomposition contains within it three sub-cases: neither $\Gamma_1$ nor $\Gamma_2$ passes over $\left(x_i, x_{i+1}\right)$, only $\Gamma_2$ passes over $\left(x_i, x_{i+1}\right)$, or only $\Gamma_1$ passes over $\left(x_i, x_{i+1}\right)$. Similarly, the limit decomposes into
\begin{equation} \label{config5}
\begin{aligned}
\lim _{x_{i+1} \rightarrow x_i}\left(x_{i+1}-x_i\right)^{6 / \kappa-1} \mathcal{J}^{(m,n)}_\alpha & =\lim _{x_{i+1} \rightarrow x_i}\left(x_{i+1}-x_i\right)^{6 / \kappa-1}\left(\left.\mathcal{J}^{(m,n)}_\alpha\right|_{\left(\Gamma_1, \Gamma_2\right) \mapsto\left(\Gamma_1^{\prime}, \Gamma_2^{\prime}\right)}\right) \\
& +\lim _{x_{i+1} \rightarrow x_i}\left(x_{i+1}-x_i\right)^{6 / \kappa-1}\left(\left.\mathcal{J}^{(m,n)}_\alpha\right|_{\left(\Gamma_1, \Gamma_2\right) \mapsto\left(\Gamma_1^{\prime \prime}, \Gamma_2^{\prime}\right)}\right)\\
& +\lim _{x_{i+1} \rightarrow x_i}\left(x_{i+1}-x_i\right)^{6 / \kappa-1}\left(\left.\mathcal{J}^{(m,n)}_\alpha\right|_{\left(\Gamma_1, \Gamma_2\right) \mapsto\left(\Gamma_1^{\prime}, \Gamma_2^{\prime \prime}\right)}\right)\\
& +\lim _{x_{i+1} \rightarrow x_i}\left(x_{i+1}-x_i\right)^{6 / \kappa-1}\left(\left.\mathcal{J}^{(m,n)}_\alpha\right|_{\left(\Gamma_1, \Gamma_2\right) \mapsto\left(\Gamma_1^{\prime \prime}, \Gamma_2^{\prime \prime}\right)}\right)
\end{aligned}
\end{equation}

The first limit on the right side of (\ref{config5}) falls under case 1 and therefore vanishes. The second (resp. third) limit on the right side of \ref{config5} falls under case 3 and therefore equals the element of $\mathcal{B}_{n-2,m-1}$ with contours $\Gamma_2^{\prime}$, (resp. $\left.\Gamma_1^{\prime}\right) \Gamma_3, \ldots, \Gamma_{m}$. Finally, the fourth limit on the right side of (\ref{config5}) still falls under case 4, and we compute it later.
There, we deform $\Gamma_1^{\prime \prime}$ and $\Gamma_2^{\prime \prime}$ in a way that generates terms only falling under cases 1 and 2. Only the latter type of term has a non-vanishing limit as $x_{i+1} \rightarrow x_i$, and a factor of $n(\kappa)^{-1}$ accompanies it. Thus, the last limit on the right side of (\ref{config5}) is the element of $\mathcal{B}_{n-2,m-1}$ with contours $\Gamma_0^{\prime}:=\left[x_{i-1}, x_{i+2}\right]^{+}, \Gamma_3, \Gamma_4, \ldots, \Gamma_{m}$. After summing all four terms, we find that the right side of (\ref{config5}) equals the element of $\mathcal{B}_{n-2,m-1}$ with contours $\Gamma:=\Gamma_0^{\prime}+\Gamma_1^{\prime}+\Gamma_2^{\prime}$, and $\Gamma_3, \Gamma_4, \ldots, \Gamma_{m}$.

\begin{equation}
   \lim_{x_i,x_{i+1}\rightarrow p} (x_{i+1}-x_i)^{\frac{6}{\kappa}-1}\mathcal{J}^{(m,n)}_\alpha(\boldsymbol{x}) =  \mathcal{J}^{(m-1,n-2)}_{\hat{\alpha}} (\boldsymbol{\hat{x}} )
\end{equation}

\end{itemize}
\end{thm}

\subsection{Case 1: Neither \texorpdfstring{$x_i$}{xi} nor \texorpdfstring{$x_{i+1}$}{xi+1} are endpoints of contours}

\begin{figure}[h]
\centering
\includegraphics[width=15cm]{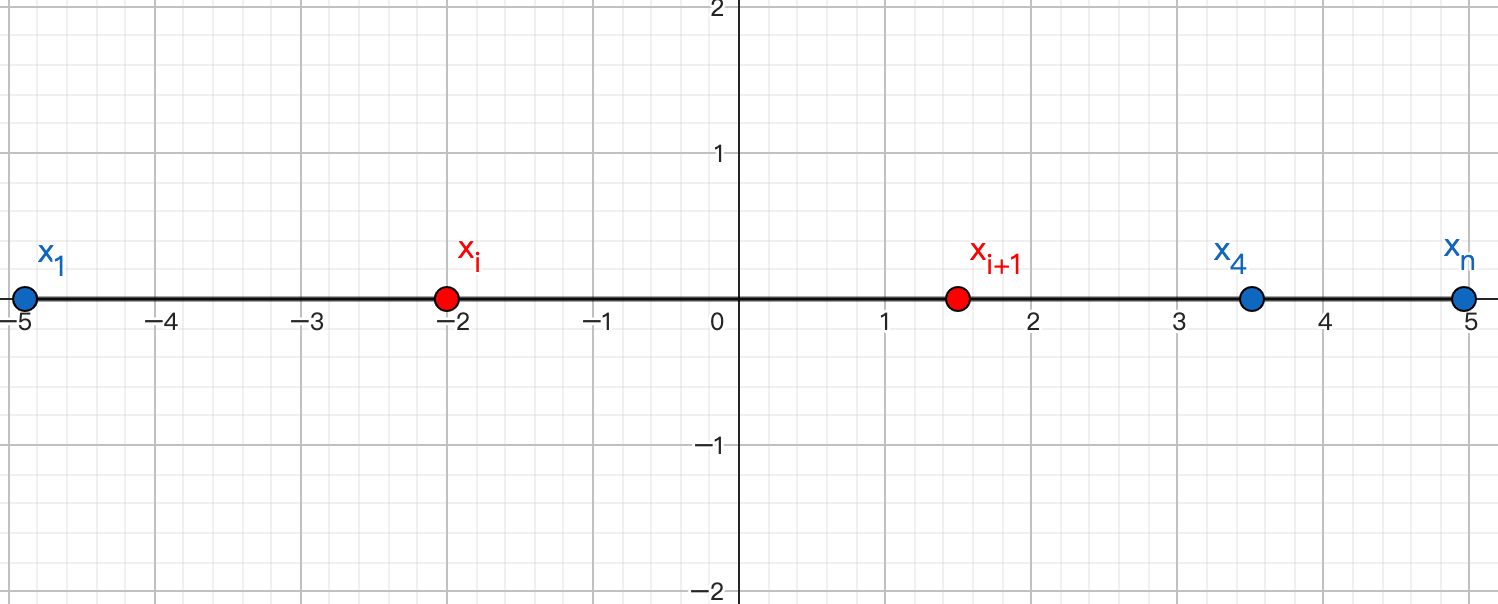}
\caption{Case 1}
\end{figure}
Suppose neither $x_i$ nor $x_{i+1}$ are endpoints of any integration contour in the Coulomb gas integral $\mathcal{J}^{(m,n)}_\alpha(\boldsymbol{x})$. Then, as $x_{i+1} \to x_i$, the renormalized Coulomb gas integral vanishes:
\[
\lim_{x_i,x_{i+1}\rightarrow p} (x_{i+1}-x_i)^{\frac{6}{\kappa}-1}\mathcal{J}^{(m,n)}_\alpha(\boldsymbol{x}) = 0.
\]

\begin{proof}
Since $x_i$ and $x_{i+1}$ are not endpoints of any integration contour, they appear only in the integrand of $\mathcal{J}^{(m,n)}_\alpha(\boldsymbol{x})$ as external marked points. The Coulomb gas integral then has the explicit representation:
\[
\begin{aligned}
(x_{i+1}-x_i)^{\frac{6}{\kappa}-1}\mathcal{J}^{(m,n)}_\alpha(\boldsymbol{x}) 
&= (x_{i+1}-x_i)^{\frac{8}{\kappa}-1} \left[\frac{n(\kappa) \Gamma(2-8 / \kappa)}{4 \sin ^2(4 \pi / \kappa) \Gamma(1-4 / \kappa)^2}\right]^m \\
&\quad \times \left( \prod_{\substack{i<j\\ j \neq i+1}}^{n}(x_j - x_i)^{2/\kappa} \right) \oint_{\Gamma_1} \cdots \oint_{\Gamma_m} \mathrm{d}u_1 \cdots \mathrm{d}u_m \\
&\quad \times \left( \prod_{k=1}^n \prod_{l=1}^m (x_k - u_l)^{-4/\kappa} \right) \left( \prod_{1 \leq p < q \leq m} (u_p - u_q)^{8/\kappa} \right).
\end{aligned}
\]

Note that the integrands remain uniformly bounded in $u_1, \dots, u_m$ as $x_{i+1} \to x_i$, since all contours $\Gamma_k$ are fixed and do not depend on $x_i$ or $x_{i+1}$, and none of them contain $x_i$ or $x_{i+1}$ as endpoints.

The only vanishing contribution comes from the prefactor $(x_{i+1} - x_i)^{\frac{8}{\kappa} - 1}$. For $\kappa \in (0,8)$, we have:
\[
\frac{8}{\kappa} - 1 > 0 \quad \Longrightarrow \quad \lim_{x_{i+1} \to x_i} (x_{i+1} - x_i)^{\frac{8}{\kappa} - 1} = 0.
\]
Thus, even though the integral remains finite, the overall expression tends to zero:
\[
\lim_{x_i,x_{i+1} \to p} (x_{i+1} - x_i)^{\frac{6}{\kappa} - 1} \mathcal{J}^{(m,n)}_\alpha(\boldsymbol{x}) = 0.
\]
\end{proof}

\subsection{Case 2: $x_i$ and $x_{i+1}$ as endpoints of a single contour}
\begin{figure}[h]
\centering
\includegraphics[width=15cm]{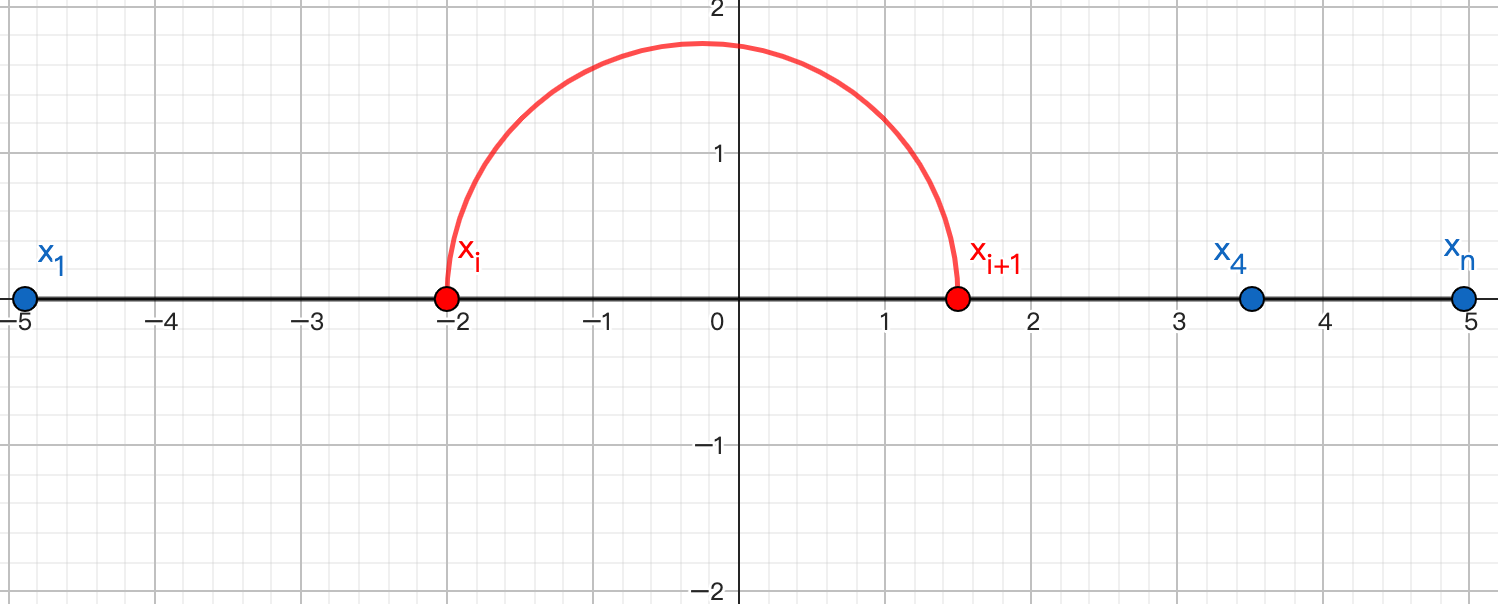}
\caption{Case 2}
\end{figure}
Suppose $x_i$ and $x_{i+1}$ are the endpoints of a single integration contour $\Gamma_1$ of the Coulomb gas integral $\mathcal{J}^{(m,n)}(\boldsymbol{x})$. In this case, the form of $\Gamma_1$ depends on the value of $\kappa$:
\[
\Gamma_1 = 
\begin{cases}
\mathscr{P}(x_i, x_{i+1}) & \text{if } 0 < \kappa \leq 4, \\
[x_i, x_{i+1}]^{+} & \text{if } 4 < \kappa < 8,
\end{cases}
\]
where $\mathscr{P}(x_i, x_{i+1})$ denotes a Pochhammer contour and $[x_i, x_{i+1}]^{+}$ is the positively oriented real interval.

We determine the asymptotic behavior of $\mathcal{J}^{(m,n)}(\boldsymbol{x})$ as $x_{i+1} \to x_i$ by analyzing the innermost integration (with respect to $u_1$) in the ordered Coulomb gas integral:
\begin{itemize}\label{steps}
    \item[1.] Use Fubini's theorem to integrate variables in order $u_1, u_2, \dots, u_m$.
    \item[2.] When $m > 1$, the limit $\lim_{x_{i+1} \to x_i} (x_{i+1} - u_m)^{-4/\kappa}$ is uniform in $u_m \in \Gamma_m$. We may thus replace $x_{i+1}$ with $x_i$ in such terms.
    \item[3.] It suffices to study the behavior of the $u_1$ integral:
    \[
    \int_{\Gamma_1} \mathcal{N}\left[\prod_{j=1}^{K}(u_1 - x_j)^{\beta_j} \right] \, \mathrm{d}u_1,
    \]
    with $K = m + n - 1$ and $\beta_j \in \left\{ -\tfrac{4}{\kappa}, \tfrac{8}{\kappa}, \tfrac{2(2n - 4m + 4 - \kappa)}{\kappa} \right\}$, where $\beta_i = \beta_{i+1} = -\tfrac{4}{\kappa}$.
    \item[4.] Push all integration contours close to the real axis, as allowed since $\gamma = 8/\kappa > 0$.
    \item[5.] Fix all other variables $u_2, \dots, u_m$ and $x_1, \dots, x_n$ as real numbers not in the interval $(x_i, x_{i+1})$.
    \item[6.] Reorder the combined set $\{x_j\}_{j=1}^n \cup \{u_j\}_{j=2}^m$ as $x_1 < x_2 < \dots < x_K$.
\end{itemize}

This yields the general integral form:
\[
\mathcal{J}^{(1,K)}(\{\beta_j\} \mid \Gamma_1 \mid x_1, \dots, x_K) 
= \int_{\Gamma_1} \mathcal{N}\left[\prod_{j=1}^K (u_1 - x_j)^{\beta_j} \right] \, \mathrm{d}u_1.
\]

We now distinguish the two regimes of $\kappa$:

\begin{itemize}
\item [Case a]: $4 < \kappa < 8$ (real contour)

Let $\Gamma_1 = [x_i, x_{i+1}]^+$. Set \( u_1(t) = (1 - t)x_i + t x_{i+1} \). The integral becomes:
\[
\mathcal{J}^{(1,K)}(\{\beta_j\} \mid [x_i, x_{i+1}]^+) 
\sim (x_{i+1} - x_i)^{\beta_i + \beta_{i+1} + 1}
\mathcal{N}\left[ \prod_{j \ne i, i+1}(x_i - x_j)^{\beta_j} \right]
\int_0^1 t^{\beta_i}(1 - t)^{\beta_{i+1}} \, \mathrm{d}t.
\]

The integral evaluates to the Euler beta function:
\[
\int_0^1 t^{\beta_i}(1 - t)^{\beta_{i+1}} \, \mathrm{d}t 
= \frac{\Gamma(\beta_i + 1)\Gamma(\beta_{i+1} + 1)}{\Gamma(\beta_i + \beta_{i+1} + 2)},
\]
so the full asymptotic is:
\[
\mathcal{J}^{(1,K)} \sim \frac{\Gamma(\beta_i + 1)\Gamma(\beta_{i+1} + 1)}{\Gamma(\beta_i + \beta_{i+1} + 2)}
(x_{i+1} - x_i)^{\beta_i + \beta_{i+1} + 1}
\mathcal{N}\left[ \prod_{j \ne i, i+1}(x_i - x_j)^{\beta_j} \right].
\]

\item[Case b] $0 < \kappa \leq 4$ 

Here the real integral diverges due to $\beta_i, \beta_{i+1} \leq -1$, so we use the Pochhammer contour $\mathscr{P}(x_i, x_{i+1})$. After the same substitution, the integral becomes:
\[
\int_{\mathscr{P}(x_i, x_{i+1})} \mathcal{N}\left[ \prod_{j=1}^K (u_1 - x_j)^{\beta_j} \right] \, \mathrm{d}u_1 
\sim (x_{i+1} - x_i)^{\beta_i + \beta_{i+1} + 1} 
\mathcal{N}\left[ \prod_{j \ne i, i+1}(x_i - x_j)^{\beta_j} \right]
\oint_{\beta(0,1)} t^{\beta_i}(1 - t)^{\beta_{i+1}} \, \mathrm{d}t.
\]

Using analytic continuation of the beta function:
\[
\oint_{\beta(0,1)} t^{\beta_i}(1 - t)^{\beta_{i+1}} \mathrm{d}t = 
4 e^{\pi i(\beta_i - \beta_{i+1})} \sin(\pi \beta_i) \sin(\pi \beta_{i+1}) 
\cdot \frac{\Gamma(\beta_i + 1)\Gamma(\beta_{i+1} + 1)}{\Gamma(\beta_i + \beta_{i+1} + 2)}.
\]

Thus, the asymptotic expansion becomes:
\[
\mathcal{J}^{(1,K)} \sim 4 e^{\pi i(\beta_i - \beta_{i+1})} \sin(\pi \beta_i) \sin(\pi \beta_{i+1})
\cdot \frac{\Gamma(\beta_i + 1)\Gamma(\beta_{i+1} + 1)}{\Gamma(\beta_i + \beta_{i+1} + 2)} 
(x_{i+1} - x_i)^{\beta_i + \beta_{i+1} + 1}
\mathcal{N}\left[ \prod_{j \ne i, i+1}(x_i - x_j)^{\beta_j} \right].
\]
\end{itemize}

This concludes the analysis for Case 2.

\subsection{Case 3: $x_i$ as contour endpoint, $x_{i+1}$ not}

\begin{figure}[h]
\centering
\includegraphics[width=15cm]{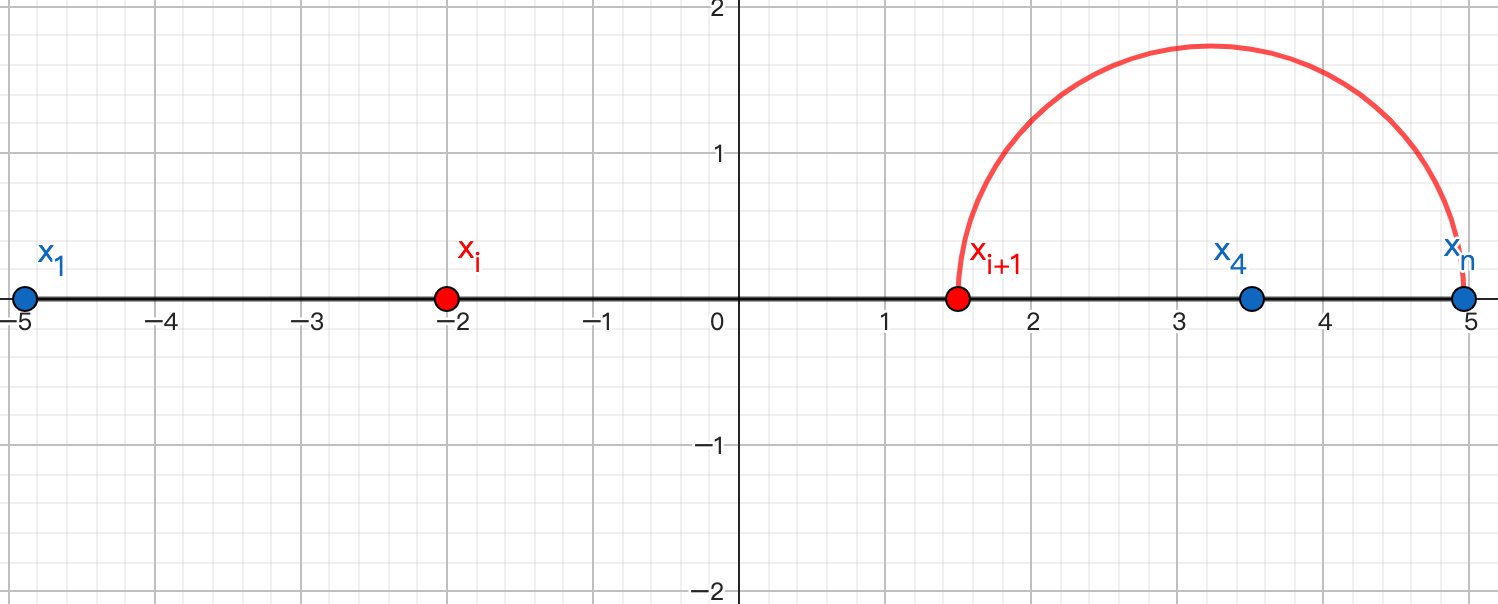}
\caption{Case 3}
\end{figure}

Either $x_i$ or $x_{i+1}$ is an endpoint of a single integration contour $\Gamma_1$ of $\mathcal{J}^{(m,n)}_\alpha(\boldsymbol{x})$, but the other is not an endpoint of any contour. We analyze the asymptotic behavior of the Coulomb gas integral $\mathcal{J}^{(m,n)}(\boldsymbol{x})$ as $x_{i+1} \rightarrow x_i$ for all $\kappa \in(0,8)$ with $8 / \kappa \notin \mathbb{Z}^{+}$.

With $\kappa>4$, $\Gamma_1$ is a simple contour, and we decompose it into one simple contour $\Gamma_1^{\prime}$ with its right endpoint at $x_{i-1}$ and another $\Gamma_1^{\prime \prime}$ with its endpoints at $x_i$ and $x_{i-1}$. (We might have $\Gamma_1^{\prime}=\emptyset$ and $\Gamma_1=\Gamma_1^{\prime \prime}$.)

After executing steps 1-6 (\ref{steps}) as shown in case 2, the key integral becomes:
\begin{equation}
\mathcal{J}^{(1, K)}\left(\{\beta_j\}\,|\,\Gamma_1^{\prime\prime}\,|\,x_1, \ldots, x_K\right) = \int_{\Gamma_1^{\prime\prime}} \mathcal{N}\left[\prod_{j=1}^\kappa (u_1 - x_j)^{\beta_j} \right] \, du_1
\end{equation}
where $K = m+n-1$, and $\mathcal{N}[\cdot]$ denotes a branch-ordering normalization for real-valuedness over the contour.

Relabeling powers $\gamma$ in $\mathcal{J}^{(m,n)}$ to $\beta_j$, we obtain:
\begin{align*}
s := \sum_{j=1}^K \beta_j = -2, &\quad \beta_j \in \left\{ -\frac{4}{\kappa}, \frac{8}{\kappa}, \frac{2(2n - 4m + 4 - \kappa)}{\kappa} \right\}, \\
\beta_{i-1} = \beta_i = \beta_{i+1} &= -\frac{4}{\kappa}, \quad \beta_{i+2} \in \left\{ -\frac{4}{\kappa}, \frac{2(2n - 4m + 4 - \kappa)}{\kappa} \right\}.
\end{align*}

These imply the milder conditions:
\begin{equation}
\sum_{j=1}^K \beta_j \in \mathbb{Z}^{-} \setminus\{-1\}, \quad \beta_j > -1, \quad \beta_i + \beta_{i+1} \notin \mathbb{Z}, \quad \beta_i + \beta_{i+1} < -1
\end{equation}

Define definite integrals:
\begin{equation}
I_k := \int_{x_k}^{x_{k+1}} \mathcal{N}\left[\prod_{j=1}^K (u_1 - x_j)^{\beta_j} \right] \, du_1, \quad \text{and } \int_{x_K}^{x_{K+1}} := \int_{x_K}^\infty + \int_{-\infty}^{x_1}
\end{equation}

Using contour deformation and the Cauchy theorem:
\begin{equation}
\sum_{k=1}^K e^{\pm \pi i \sum_{l=1}^k \beta_l} I_k = 0
\end{equation}

Solving for $I_{i-1}$ gives:
\begin{align*}
I_{i-1} &\sim -\sum_{k=1}^{i-2} \frac{\sin(\pi \sum_{l=k+1}^{i+1} \beta_l)}{\sin(\pi(\beta_i+\beta_{i+1}))} \int_{x_k}^{x_{k+1}} \mathcal{N}\left[(u-x_i)^{\beta_i+\beta_{i+1}} \prod_{j \neq i,i+1} (u-x_j)^{\beta_j} \right] du \\
&\quad + \sum_{k=i+2}^{K} \frac{\sin(\pi \sum_{l=i+2}^{k} \beta_l)}{\sin(\pi(\beta_i+\beta_{i+1}))} \int_{x_i}^{x_{k+1}} \mathcal{N}\left[(u-x_i)^{\beta_i+\beta_{i+1}} \prod_{j \neq i,i+1} (u-x_j)^{\beta_j} \right] du \\
&\quad - \frac{\sin(\pi \beta_{i+1}) \Gamma(\beta_i + 1) \Gamma(\beta_{i+1} + 1)}{\sin(\pi(\beta_i + \beta_{i+1})) \Gamma(\beta_i + \beta_{i+1} + 2)} (x_{i+1} - x_i)^{\beta_i + \beta_{i+1} + 1} \prod_{j \neq i,i+1} (x_i - x_j)^{\beta_j}
\end{align*}

Therefore,
\begin{equation}
\mathcal{J}^{(1,K)} \sim - \frac{\sin(\pi \beta_{i+1}) \Gamma(\beta_i + 1) \Gamma(\beta_{i+1} + 1)}{\sin(\pi(\beta_i + \beta_{i+1})) \Gamma(\beta_i + \beta_{i+1} + 2)} (x_{i+1} - x_i)^{\beta_i + \beta_{i+1} + 1} \prod_{j \neq i,i+1} (x_i - x_j)^{\beta_j}
\end{equation}

To analytically continue to $\kappa \in (0,4]$, where $\beta_k \leq -1$, we replace each definite integral by its Pochhammer contour analog:
\begin{equation}
\int_{x_k}^{x_{k+1}} \longmapsto \frac{1}{4 e^{\pi i(\beta_k - \beta_{k+1})} \sin(\pi \beta_k) \sin(\pi \beta_{k+1})} \oint_{\mathscr{P}(x_k,x_{k+1})}
\end{equation}

This yields for $\kappa \in (0,4]$:
\begin{align*}
\mathcal{J}^{(1,K)} \sim &\ 4 e^{\pi i(\beta_{i-1} - \beta_i)} \sin(\pi \beta_{i-1}) \sin(\pi \beta_i)
\cdot \left[ - \frac{\sin(\pi \beta_{i+1}) \Gamma(\beta_i + 1) \Gamma(\beta_{i+1} + 1)}{\sin(\pi(\beta_i + \beta_{i+1})) \Gamma(\beta_i + \beta_{i+1} + 2)} \right] \\
& \cdot (x_{i+1} - x_i)^{\beta_i + \beta_{i+1} + 1} \prod_{j \neq i,i+1} (x_i - x_j)^{\beta_j}
\end{align*}

\subsection{Case 4: $x_i$, $x_{i+1}$ are endpoints of two distinct contour} 
\begin{figure}[h]
\centering
\includegraphics[width=15cm]{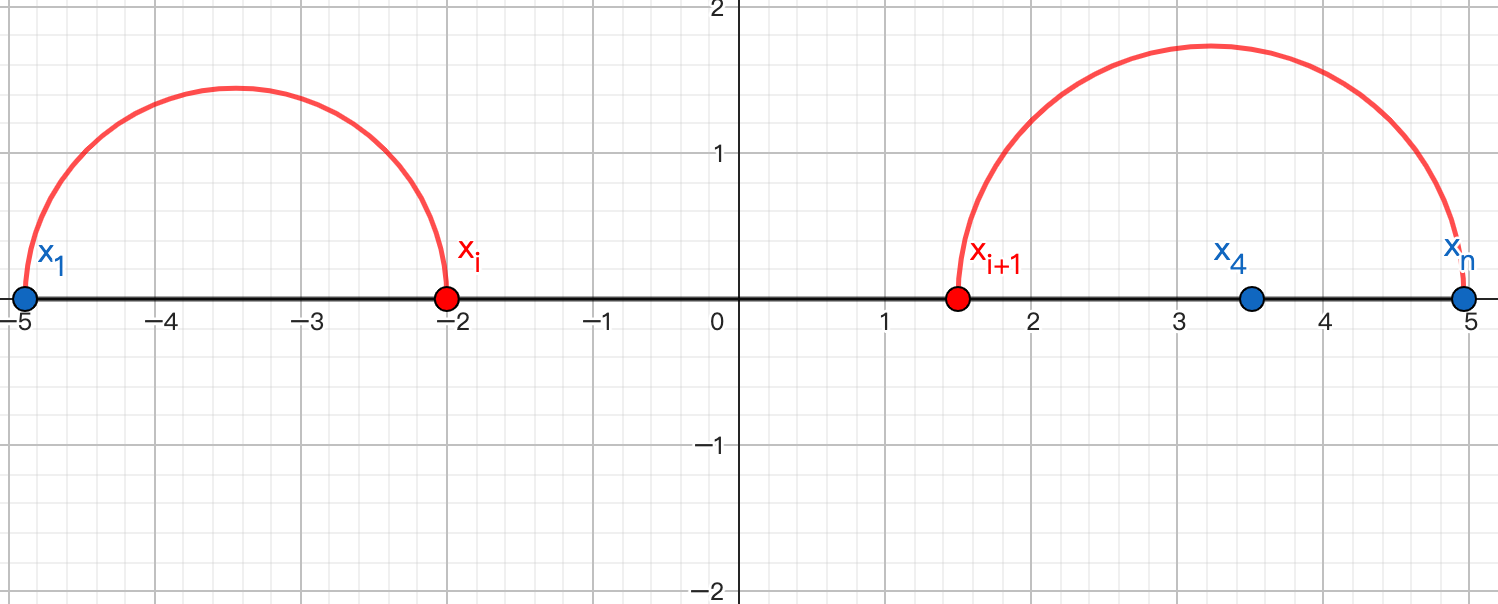}
\caption{Case 4}
\end{figure}
We analyze the asymptotic behavior of the Coulomb gas integral $\mathcal{J}^{(m,n)}(\boldsymbol{x})$ as $x_{i+1}\to x_i$ when $x_i$ is an endpoint of contour $\Gamma_1$ and $x_{i+1}$ is an endpoint of a different contour $\Gamma_2$. For $\kappa>4$ with $\Gamma_1$ and $\Gamma_2$ being simple contours, we decompose:

\[
\Gamma_1 = \begin{cases}
\Gamma_1' \text{ (endpoints at $x_{i-1}$)} \\
\Gamma_1'' \text{ (endpoints at $x_{i-1},x_i$)}
\end{cases}, \quad
\Gamma_2 = \begin{cases}
\Gamma_2' \text{ (endpoints at $x_{i+2}$)} \\
\Gamma_2'' \text{ (endpoints at $x_{i+1},x_{i+2}$)}
\end{cases}
\]

where $\Gamma_1'$ or $\Gamma_2'$ may be empty. The exact form depends on $\kappa$:

\[
\Gamma_1'' = \begin{cases}
\mathscr{P}(x_{i-1},x_i), & 0<\kappa\leq4 \\
[x_{i-1},x_i]^+, & 4<\kappa<8
\end{cases}, \quad
\Gamma_2'' = \begin{cases}
\mathscr{P}(x_{i+1},x_{i+2}), & 0<\kappa\leq4 \\
[x_{i+1},x_{i+2}]^+, & 4<\kappa<8
\end{cases}
\]

Using Fubini's theorem, we first integrate $u_1$ along $\Gamma_1$, then $u_2$ along $\Gamma_2$, followed by $u_3,\ldots,u_m$ along their respective contours. For $l>2$, $(x_{i+1}-u_l)^{-4/\kappa}\to(x_i-u_l)^{-4/\kappa}$ uniformly as $x_{i+1}\to x_i$, so the asymptotic behavior is determined by:

\[
\mathcal{J}^{(2,K)}\left(\{\beta_j\};\gamma|\Gamma_1'',\Gamma_2''|x_1,\ldots,x_K\right) = \iint_{\Gamma_1''\times\Gamma_2''} \mathcal{N}\left[\prod_{j=1}^K (u_1-x_j)^{\beta_j}(u_2-x_j)^{\beta_j}(u_2-u_1)^\gamma\right]du_2 du_1
\]

where $K=m+n-2$ and $x_j\notin(x_{i-1},x_{i+2})$. The exponents satisfy:

\[
\begin{gathered}
s := \sum_{j=1}^K \beta_j + \gamma = -2 \\
\beta_j \in \{-4/\kappa, 8/\kappa, 2(2n-4m+4-\kappa)/\kappa\} \\
\beta_i = \beta_{i+1} = -4/\kappa, \quad \beta_{i+2} \in \{-4/\kappa, 2(2n-4m+4-\kappa)/\kappa\}, \quad \gamma = 8/\kappa
\end{gathered}
\]

For $\kappa\in(4,8)$, these imply $s \in \mathbb{Z}^{-}\setminus\{-1\}$, $\beta_i+\beta_{i+1}+\gamma/2 \notin\mathbb{Z}$, $\beta_j > -1$ for all $j$, $\beta_i+\beta_{i+1} < -1$, and $\beta_i+\gamma/2 = 0$.

We define the integrals for $m,k\in\{1,\ldots,K\}$:

\[
I_{m,k}(\boldsymbol{x}) := \int_{x_m}^{x_{m+1}}\int_{x_k}^{x_{k+1}} \mathcal{N}[\cdots] du_2 du_1
\]

with the convention for $k=K$:

\[
\int_{x_K}^{x_{K+1}} := \int_{x_K}^\infty + \int_{-\infty}^{x_1}
\]

For the diagonal terms ($m=k$):

\[
I_{k,k}(\boldsymbol{x}) := \int_{x_{k-1}}^{x_k}\int_{x_{k-1}}^{u_1} \mathcal{N}[\cdots] du_2 du_1 + \int_{x_{k-1}}^{x_k}\int_{u_1}^{x_k} \mathcal{N}[\cdots] du_2 du_1
\]

Using a large semicircular contour (radius $R\to\infty$) in the upper/lower half-plane:

\[
\sum_{k=1}^{i-2} e^{\pm\pi i\sum_{l=1}^k \beta_l} I_{i-1,k} + e^{\pm\pi i\sum_{l=1}^{i-1}\beta_l}(1+e^{\pm\pi i\gamma})I_{i-1,i-1} + \sum_{k=i}^K e^{\pm\pi i(\sum_{l=1}^k \beta_l+\gamma)} I_{i-1,k} = 0
\]

Solving these equations yields:

\begin{equation} 
\begin{aligned}
I_{i-1,i+1} = & \sum_{k=1}^{i-2} \frac{\sin\pi(\sum_{l=k+1}^{i-1}\beta_l+\gamma/2)}{\sin\pi(\beta_i+\beta_{i+1}+\gamma/2)} I_{i-1,k} \\
& - \sum_{k=i+2}^K \frac{\sin\pi(\sum_{l=i}^k\beta_l+\gamma/2)}{\sin\pi(\beta_i+\beta_{i+1}+\gamma/2)} I_{i-1,k} \\
& - \frac{\sin\pi(\beta_i+\gamma/2)}{\sin\pi(\beta_i+\beta_{i+1}+\gamma/2)} I_{i-1,i}
\end{aligned}
\end{equation}

For $k=i$, $I_{i-1,i}$ contains divergent terms ($\beta_i+\beta_{i+1}<-1$) but its coefficient vanishes due to $\beta_i+\gamma/2=0$. For $k\notin\{i,i\pm1\}$:

\[
I_{i-1,k} \sim -\frac{\sin\pi\beta_{i+1}}{\sin\pi(\beta_i+\beta_{i+1})} I_{i,k}
\]

where $I_{i,k}$ falls under Case 2 and diverges as $(x_{i+1}-x_i)^{\beta_i+\beta_{i+1}+1}$.

After careful analysis of all terms, we obtain:

\[
\begin{aligned}
I_{i-1,i+1} &\sim -\frac{\sin\pi\beta_{i+1}\Gamma(\beta_i+1)\Gamma(\beta_{i+1}+1)}{\sin\pi(\beta_i+\beta_{i+1})\Gamma(\beta_i+\beta_{i+1}+2)} (x_{i+1}-x_i)^{\beta_i+\beta_{i+1}+1} \\
&\quad \times \mathcal{N}\left[\prod_{j\neq i,i+1}^K (x_i-x_j)^{\beta_j}\right] \\
&\quad \times \left[\sum_{k=1}^{i-2} \frac{\sin\pi(\sum_{l=k+1}^{i-1}\beta_l+\gamma/2)}{\sin\pi(\beta_i+\beta_{i+1}+\gamma/2)} - \sum_{k=i+2}^K \frac{\sin\pi(\sum_{l=i}^k\beta_l+\gamma/2)}{\sin\pi(\beta_i+\beta_{i+1}+\gamma/2)}\right] \\
&\quad \times \int_{x_k}^{x_{k+1}} \mathcal{N}\left[(u_2-x_i)^{\beta_i+\beta_{i+1}+\gamma} \prod_{j\neq i,i+1}^K (u_2-x_j)^{\beta_j}\right] du_2
\end{aligned}
\]

Introducing the joined contour $\Gamma_0' = [x_{i-1},x_{i+2}]^+$, we obtain the cleaner asymptotic form:

\[
\begin{aligned}
\mathcal{J}^{(2,K)} &\sim \frac{\sin\pi\beta_{i+1}\Gamma(\beta_i+1)\Gamma(\beta_{i+1}+1)}{\sin\pi(\beta_i+\beta_{i+1})\Gamma(\beta_i+\beta_{i+1}+2)} (x_{i+1}-x_i)^{\beta_i+\beta_{i+1}+1} \\
&\quad \times \mathcal{N}\left[\prod_{j\neq i,i+1}^K (x_i-x_j)^{\beta_j}\right] \int_{x_{i-1}}^{x_{i+2}} \mathcal{N}\left[\prod_{j\neq i,i+1}^K (u_2-x_j)^{\beta_j}\right] du_2
\end{aligned}
\]

For $\kappa\leq4$, we use the analytically continued form:

\[
\int_{x_k}^{x_{k+1}} \mapsto \frac{1}{4e^{\pi i(\beta_k-\beta_{k+1})}\sin\pi\beta_k\sin\pi\beta_{k+1}} \oint_{\mathscr{P}(x_k,x_{k+1})}
\]

which yields:

\[
\begin{aligned}
\mathcal{J}^{(2,K)} &\sim 4e^{\pi i(\beta_{i-1}-\beta_i)}\sin\pi\beta_{i-1}\sin\pi\beta_i \\
&\quad \times \frac{\sin\pi\beta_{i+1}\Gamma(\beta_i+1)\Gamma(\beta_{i+1}+1)}{\sin\pi(\beta_i+\beta_{i+1})\Gamma(\beta_i+\beta_{i+1}+2)} (x_{i+1}-x_i)^{\beta_i+\beta_{i+1}+1} \\
&\quad \times \mathcal{N}\left[\prod_{j\neq i,i+1}^K (x_i-x_j)^{\beta_j}\right]
\end{aligned}
\]

Combining all cases, the final asymptotic behavior is:

\[
\lim_{x_i,x_{i+1}\to p} (x_{i+1}-x_i)^{6/\kappa-1} \mathcal{J}^{(m,n)}_\alpha(\boldsymbol{x}) = \mathcal{J}^{(m-1,n-2)}_{\hat{\alpha}}(\boldsymbol{\hat{x}})
\]

\section{Asymptotics of radial partition functions as interval collapse}

In this section, we compute the asymptotic behavior of the following Coulomb gas integral $\mathcal{J}^{(m,n)}_\alpha(\boldsymbol{x})$ following the method developed in \cite{FK15a}.
\begin{equation}
\begin{aligned}
\mathcal{J}^{(m,n)}_\alpha(\boldsymbol{x}) =&\left[\frac{n(\kappa) \Gamma(2-8 / \kappa)}{4 \sin ^2(4 \pi / \kappa) \Gamma(1-4 / \kappa)^2}\right]^m \left(\prod_{i<j}^{n}\left(x_j-x_i\right)^{2 / \kappa}\right) \\
& \left(\prod_{i=1}^{n}\left|z-x_i\right|^{(\kappa-2n+4m-4) / \kappa}\right) \times|z-\bar{z}|^{(\kappa-2n+4m-4)^2 / 8 \kappa} \\
&\oint_{\Gamma_1} \ldots \oint_{\Gamma_{m}} d u_1 \ldots d u_{m}\left(\prod_{k=1}^{n} \prod_{l=1}^{m}\left(x_k-u_l\right)^{-4 / \kappa}\right)\left(\prod_{p<q}^{m}\left(u_p-u_q\right)^{8 / \kappa}\right) \\
&
\times\left(\prod_{k=1}^{m}\left(z-u_k\right)^{(2n-4m+4-\kappa) / \kappa}\left(\bar{z}-u_k\right)^{(2n-4m+4-\kappa) / \kappa}\right)
\end{aligned}
\end{equation}

\begin{thm}
\indent
\begin{itemize}
 \item[Configuration 1] Neither $x_i$ nor $x_{i+1}$ are endpoints of an integration contour of $\mathcal{J}^{(m,n)}_\alpha(\boldsymbol{x})$.

\begin{equation}
   \lim_{x_i,x_{i+1}\rightarrow p} (x_{i+1}-x_i)^{\frac{6}{\kappa}-1}\mathcal{J}^{(m,n)}_\alpha(\boldsymbol{x}) = 0   
\end{equation}

\item[Configuration 2] Both $x_i$ and $x_{i+1}$ are endpoints of a single, common integration contour $\Gamma_1$ of $\alpha$. Hence, this contour is $\Gamma_1=\mathscr{P}\left(x_i, x_{i+1}\right)$ or $\left[x_i, x_{i+1}\right]^{+}$. (The superscript + indicates that we form the contour $\left[x_i, x_{i+1}\right]^{+}$by slightly bending $\left[x_i, x_{i+1}\right]$ into the upper half-plane, keeping the endpoints fixed.)

\begin{equation}
   \lim_{x_i,x_{i+1}\rightarrow p} (x_{i+1}-x_i)^{\frac{6}{\kappa}-1}\mathcal{J}^{(m,n)}_\alpha(\boldsymbol{x}) = n(\kappa) \mathcal{J}^{(m,n)}_{\hat{\alpha}} (\boldsymbol{\hat{x}} )
\end{equation}

\item[Configuration 3] Either $x_i$ or $x_{i+1}$ is an endpoint of a single integration contour $\Gamma_1$ of $\mathcal{J}^{(m,n)}_\alpha(\boldsymbol{x})$, but the other is not an endpoint of any contour. 

We assume that $i>1, \kappa>4, x_i$ is an endpoint of $\Gamma_1$, and $\Gamma_1$ does not pass over the interval $\left(x_i, x_{i+1}\right)$.

With $\kappa>4$, $\Gamma_1$ is a simple contour, and we decompose it into one simple contour $\Gamma_1^{\prime}$ with its right endpoint at $x_{i-1}$ and another $\Gamma_1^{\prime \prime}$ with its endpoints at $x_i$ and $x_{i-1}$. (We might have $\Gamma_1^{\prime}=\emptyset$ and $\Gamma_1=\Gamma_1^{\prime \prime}$.) Also, the limit breaks into
\begin{equation} \label{config3}
\begin{aligned}
\lim _{x_{i}, x_{i+1}\rightarrow p}\left(x_{i+1}-x_i\right)^{6 / \kappa-1} \mathcal{J}^{(m,n)}_\alpha= & \lim _{x_{i+1} \rightarrow x_i}\left(x_{i+1}-x_i\right)^{6 / \kappa-1}\left(\left.\mathcal{J}^{(m,n)}_\alpha\right|_{\Gamma_1 \mapsto \Gamma_1^{\prime}}\right) \\
& +\lim _{x_{i+1} \rightarrow x_i}\left(x_{i+1}-x_i\right)^{6 / \kappa-1}\left(\left.\mathcal{J}^{(m,n)}_\alpha\right|_{\Gamma_1 \mapsto \Gamma_1^{\prime \prime}}\right) .
\end{aligned}
\end{equation}
The first limit on the right side of \ref{config3} falls under case 1 and, therefore, vanishes. Meanwhile, the second limit on the right side of \ref{config3}  still falls under case 3. 

\begin{equation}
   \lim_{x_i,x_{i+1}\rightarrow p} (x_{i+1}-x_i)^{\frac{6}{\kappa}-1}\mathcal{J}^{(m,n)}_\alpha(\boldsymbol{x}) = \mathcal{J}^{(m-1,n)}_{\hat{\alpha}} (\boldsymbol{\hat{x}} )
\end{equation}
   
\item[Configuration 4]  $x_i$ is an endpoint of one contour $\Gamma_1$ , and $x_{i+1}$ is an endpoint of a different contour $\Gamma_2$.

With $\kappa>4$ and $\Gamma_1$ and $\Gamma_2$ simple contours, we decompose $\Gamma_1$ (resp. $\Gamma_2$ ) into one simple contour $\Gamma_1^{\prime}$ (resp. $\Gamma_2^{\prime}$ ) with an endpoint at $x_{i-1}$ (resp. $x_{i+2}$ ) and another $\Gamma_1^{\prime \prime}$ (resp. $\Gamma_2^{\prime \prime}$ ) with its endpoints at $x_{i-1}$ and $x_i$ (resp. $x_{i+1}$ and $x_{i+2}$ ). In some cases, we might have $\Gamma_1^{\prime}=\emptyset$ and $\Gamma_1^{\prime \prime}=\Gamma_1$ (resp. $\Gamma_2^{\prime}=\emptyset$ and $\Gamma_2^{\prime \prime}=\Gamma_2$ ). This decomposition contains within it three sub-cases: neither $\Gamma_1$ nor $\Gamma_2$ passes over $\left(x_i, x_{i+1}\right)$, only $\Gamma_2$ passes over $\left(x_i, x_{i+1}\right)$, or only $\Gamma_1$ passes over $\left(x_i, x_{i+1}\right)$. Similarly, the limit decomposes into
\begin{equation}
\begin{aligned}
\lim _{x_{i}, x_{i+1}\rightarrow p}\left(x_{i+1}-x_i\right)^{6 / \kappa-1} \mathcal{J}^{(m,n)}_\alpha & =\lim _{x_{i+1} \rightarrow x_i}\left(x_{i+1}-x_i\right)^{6 / \kappa-1}\left(\left.\mathcal{J}^{(m,n)}_\alpha\right|_{\left(\Gamma_1, \Gamma_2\right) \mapsto\left(\Gamma_1^{\prime}, \Gamma_2^{\prime}\right)}\right) \\
& +\lim _{x_{i+1} \rightarrow x_i}\left(x_{i+1}-x_i\right)^{6 / \kappa-1}\left(\left.\mathcal{J}^{(m,n)}_\alpha\right|_{\left(\Gamma_1, \Gamma_2\right) \mapsto\left(\Gamma_1^{\prime \prime}, \Gamma_2^{\prime}\right)}\right)\\
& +\lim _{x_{i+1} \rightarrow x_i}\left(x_{i+1}-x_i\right)^{6 / \kappa-1}\left(\left.\mathcal{J}^{(m,n)}_\alpha\right|_{\left(\Gamma_1, \Gamma_2\right) \mapsto\left(\Gamma_1^{\prime}, \Gamma_2^{\prime \prime}\right)}\right)\\
& +\lim _{x_{i+1} \rightarrow x_i}\left(x_{i+1}-x_i\right)^{6 / \kappa-1}\left(\left.\mathcal{J}^{(m,n)}_\alpha\right|_{\left(\Gamma_1, \Gamma_2\right) \mapsto\left(\Gamma_1^{\prime \prime}, \Gamma_2^{\prime \prime}\right)}\right)
\end{aligned}
\end{equation}

The first limit on the right side of (\ref{config5}) falls under case 1 and therefore vanishes. The second (resp. third) limit on the right side of \ref{config5} falls under case 3 and therefore equals the element of $\mathcal{B}_{n-2,m-1}$ with contours $\Gamma_2^{\prime}$, (resp. $\left.\Gamma_1^{\prime}\right) \Gamma_3, \ldots, \Gamma_{m}$. Finally, the fourth limit on the right side of (\ref{config5}) still falls under case 4, and we compute it later.
There, we deform $\Gamma_1^{\prime \prime}$ and $\Gamma_2^{\prime \prime}$ in a way that generates terms only falling under cases 1 and 2. Only the latter type of term has a non-vanishing limit as $x_{i+1} \rightarrow x_i$, and a factor of $n(\kappa)^{-1}$ accompanies it. Thus, the last limit on the right side of (\ref{config5}) is the element of $\mathcal{B}_{n-2,m-1}$ with contours $\Gamma_0^{\prime}:=\left[x_{i-1}, x_{i+2}\right]^{+}, \Gamma_3, \Gamma_4, \ldots, \Gamma_{m}$. After summing all four terms, we find that the right side of (\ref{config5}) equals the element of $\mathcal{B}_{n-2,m-1}$ with contours $\Gamma:=\Gamma_0^{\prime}+\Gamma_1^{\prime}+\Gamma_2^{\prime}$, and $\Gamma_3, \Gamma_4, \ldots, \Gamma_{m}$.

\begin{equation}
   \lim_{x_i,x_{i+1}\rightarrow p} (x_{i+1}-x_i)^{\frac{6}{\kappa}-1}\mathcal{J}^{(m,n)}_\alpha(\boldsymbol{x}) =  \mathcal{J}^{(m-1,n-2)}_{\hat{\alpha}} (\boldsymbol{\hat{x}} )
\end{equation}
\item[Configuration 5]When \( x_i \) and \( x_{i+1} \) are endpoints of the same integration contour \( \Gamma_1 \), but we consider their collapse along the complementary arc. 

Without loss of generality, we may assume that \( x_1 = -R \) and \( x_n = R \) for large \( R \), and study the asymptotic behavior as \( R \to \infty \). In this case, the Coulomb gas integral \( \mathcal{J}^{(m,n)}_\alpha(\boldsymbol{x}) \) admits the following decomposition:
\begin{equation} \label{eq:case5-decomposition}
\begin{split}
\lim_{R\to\infty} (2R)^{6/\kappa - 1} \mathcal{J}^{(m,n)}_\alpha(\boldsymbol{x}) = \;& 
\lim_{R\to\infty} (2R)^{6/\kappa - 1} \left.\mathcal{J}^{(m,n)}_\alpha\right|_{\Gamma_1 \to \Gamma_1'} \\
& + \lim_{R\to\infty} (2R)^{6/\kappa - 1} \left.\mathcal{J}^{(m,n)}_\alpha\right|_{\Gamma_1 \to \Gamma_1''} \\
& + \lim_{R\to\infty} (2R)^{6/\kappa - 1} \left.\mathcal{J}^{(m,n)}_\alpha\right|_{\Gamma_1 \to \Gamma_1'''}.
\end{split}
\end{equation}

With $\kappa>4$ and \( \Gamma_1 \) simple contours, we decompose $\Gamma_1$ as \( \Gamma_1' \): an arc connecting \( (x_1, x_2) \),\( \Gamma_1'' \): the portion of \( \Gamma_1 \) running from \( x_2 \) to \( x_{n-1} \), \( \Gamma_1''' \): the arc connecting \( (x_{n-1}, x_n) \).

The same decomposition applies in the general setting of collapsing \( x_i \) and \( x_{i+1} \), yielding:
\begin{equation}
\begin{aligned}
\lim _{x_{i}, x_{i+1}\rightarrow p} (x_{i+1} - x_i)^{6/\kappa - 1} \mathcal{J}^{(m,n)}_\alpha(\boldsymbol{x}) = \;& 
\lim_{x_{i+1} \to x_i} (x_{i+1} - x_i)^{6/\kappa - 1} \left.\mathcal{J}^{(m,n)}_\alpha\right|_{\Gamma_1 \mapsto \Gamma_1'} \\
& + \lim_{x_{i+1} \to x_i} (x_{i+1} - x_i)^{6/\kappa - 1} \left.\mathcal{J}^{(m,n)}_\alpha\right|_{\Gamma_1 \mapsto \Gamma_1''} \\
& + \lim_{x_{i+1} \to x_i} (x_{i+1} - x_i)^{6/\kappa - 1} \left.\mathcal{J}^{(m,n)}_\alpha\right|_{\Gamma_1 \mapsto \Gamma_1'''}.
\end{aligned}
\end{equation}

In this complementary collapse configuration, we eventually obtain the expected fusion behavior:
\begin{equation}
\lim_{x_i, x_{i+1} \to p} (x_{i+1} - x_i)^{6/\kappa - 1} \mathcal{J}^{(m,n)}_\alpha(\boldsymbol{x}) = 2 \mathcal{J}^{(m-1,n)}_{\hat{\alpha}}(\hat{\boldsymbol{x}}),
\end{equation}
where \( \hat{\alpha} \) is the link pattern obtained by removing the arc \( (x_i, x_{i+1}) \), and \( \hat{\boldsymbol{x}} \) is the configuration with these two points collapsed to \( p \).

\end{itemize}
\end{thm}

\subsection{Case 1: Neither \texorpdfstring{$x_i$}{xi} nor \texorpdfstring{$x_{i+1}$}{xi+1} are endpoints of contours} \label{case1}
\begin{figure}[h]
\centering
\includegraphics[width=15cm]{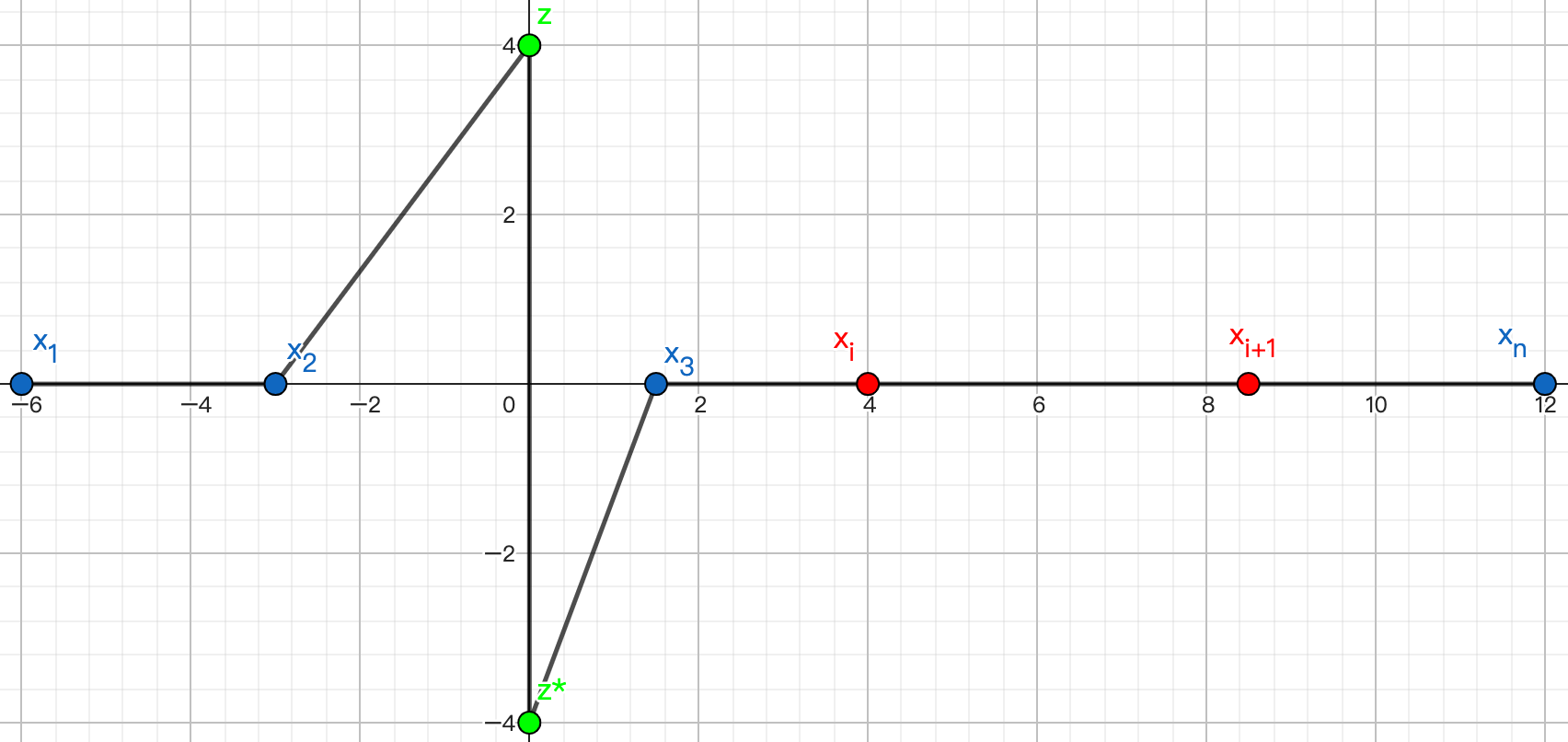}
\caption{Case 1}
\end{figure}
When neither $x_i$ nor $x_{i+1}$ are endpoints of any integration contour in $\mathcal{J}_{\alpha}^{(m,n)}(\boldsymbol{x})$, the asymptotic behavior is given by:

\begin{equation}
\begin{aligned}
&(x_{i+1}-x_i)^{\frac{6}{\kappa}-1}\mathcal{J}^{(m,n)}_\alpha(\boldsymbol{x}) \\
&= (x_{i+1}-x_i)^{\frac{8}{\kappa}-1} 
\left(\prod_{k=1}^{m} \frac{n(\kappa) \Gamma(2-8/\kappa)}{4 e^{\pi i(\beta_{1k}-\beta_{2k})} \sin(\pi \beta_{1k}) \sin(\pi \beta_{2k}) \Gamma(1-4/\kappa)^2}\right) \\
&\times \left(\prod_{\substack{i<j\\ j\neq i+1}}^{n}(x_j-x_i)^{2/\kappa}\right) 
\left(\prod_{i=1}^{n}|z-x_i|^{(\kappa-2n+4m-4)/\kappa}\right) 
|z-\bar{z}|^{(\kappa-2n+4m-4)^2/8\kappa} \\
&\times \oint_{\Gamma_1} \cdots \oint_{\Gamma_{m}} du_1 \cdots du_{m}
\left(\prod_{k=1}^{n} \prod_{l=1}^{m}(x_k-u_l)^{-4/\kappa}\right)
\left(\prod_{p<q}^{m}(u_p-u_q)^{8/\kappa}\right) \\
&\times \left(\prod_{k=1}^{m}(z-u_k)^{(2n-4m+4-\kappa)/\kappa}(\bar{z}-u_k)^{(2n-4m+4-\kappa)/\kappa}\right)
\end{aligned}
\end{equation}

The limit as $x_{i+1} \rightarrow x_i$ of $(x_{i+1}-u_m)^{-4/\kappa}$ is uniform over $u_m \in \Gamma_m$. For $\kappa \in (0,8)$, we have $\frac{8}{\kappa}-1 > 0$, which implies:

\begin{equation}
\lim_{x_i,x_{i+1}\rightarrow p} (x_{i+1}-x_i)^{\frac{8}{\kappa}-1} = 0.
\end{equation}

Consequently, we can directly set $x_{i+1}=x_i=p$ in the Coulomb gas integral, yielding:

\begin{equation}
\lim_{x_i,x_{i+1}\rightarrow p} (x_{i+1}-x_i)^{\frac{6}{\kappa}-1}\mathcal{J}^{(m,n)}_\alpha(\boldsymbol{x}) = 0.
\end{equation}

\begin{remark}
Key observations are:
\begin{itemize}
\item The vanishing limit occurs because $\frac{8}{\kappa}-1 > 0$ for all $\kappa \in (0,8)$
\item The uniformity of the limit allows direct substitution in the integral
\item The result shows no singularity develops when neither point is a contour endpoint
\end{itemize}
\end{remark}
\subsection{Case 2: $x_i$ and $x_{i+1}$ as endpoints of a single Contour} \label{case2}
\begin{figure}[h]
\centering
\includegraphics[width=15cm]{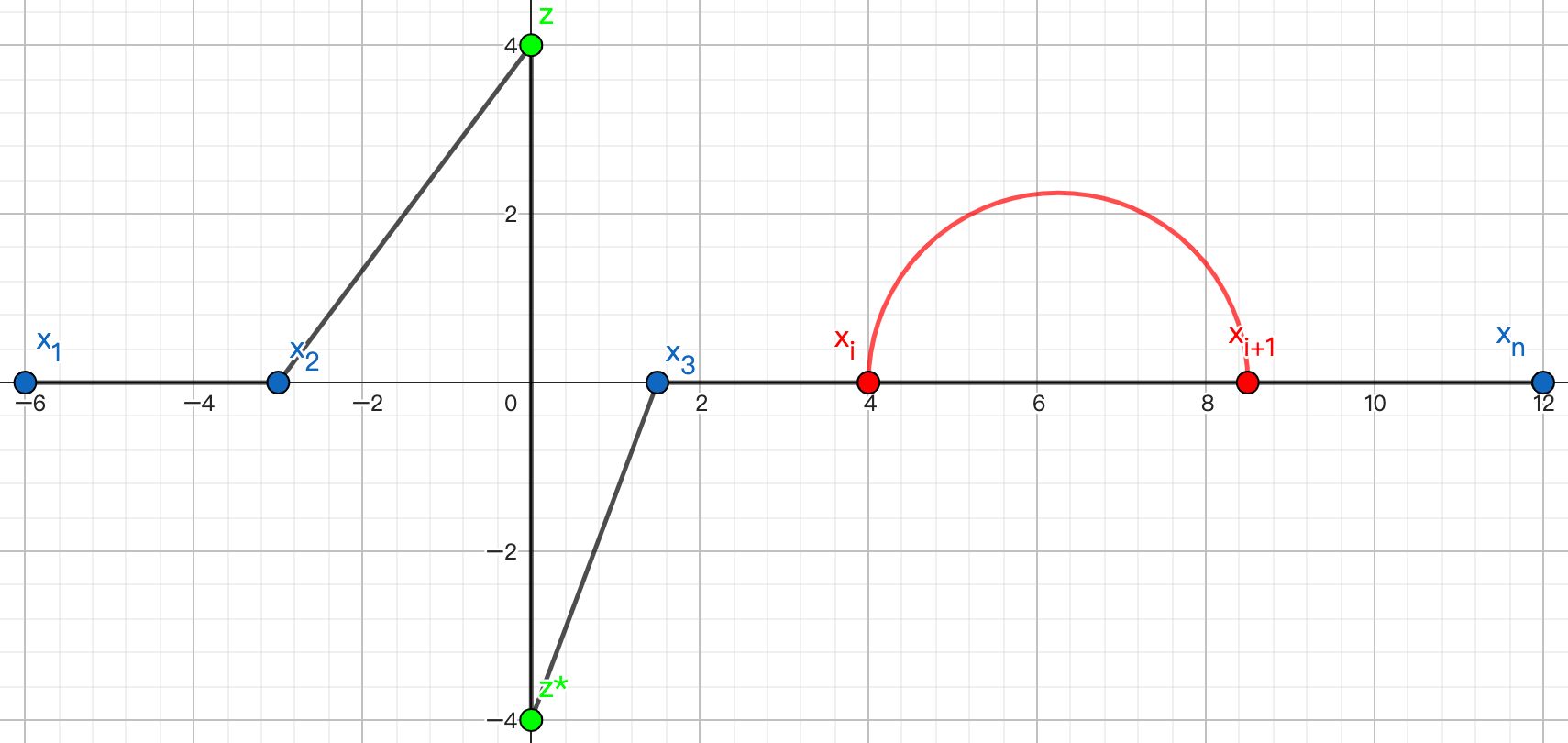}
\caption{Case 2}
\end{figure}
When both $x_i$ and $x_{i+1}$ are endpoints of the same integration contour $\Gamma_1$ in $\mathcal{J}^{(m,n)}_\alpha(\boldsymbol{x})$, the asymptotic behavior is governed by the contour structure:

\[
\Gamma_1 = \begin{cases}
\mathscr{P}(x_i,x_{i+1}) & \text{for } 0<\kappa\leq4 \\
[x_i,x_{i+1}]^+ & \text{for } 4<\kappa<8
\end{cases}
\]
where $[x_i,x_{i+1}]^+$ denotes a slight deformation into the upper half-plane and $\mathscr{P}(x_i,x_{i+1})$ is the Pochhammer contour.

The analysis proceeds through the following steps:

\begin{enumerate}
\item Using Fubini's theorem, we integrate $u_1$ first along $\Gamma_1$, followed by $u_2,\ldots,u_m$ along their respective contours.

\item For $m>1$, the limit $(x_{i+1}-u_l)^{-4/\kappa} \to (x_i-u_l)^{-4/\kappa}$ is uniform when $u_l \in \Gamma_l$ ($l>1$), allowing us to set $x_{i+1}=x_i$ in these factors.

\item The dominant behavior comes from the $u_1$ integration:
\[
\mathcal{J}^{(1,K)}\left(\{\beta_j\}|\Gamma_1|x_1,\ldots,x_K\right) = \int_{\Gamma_1} \mathcal{N}\left[\prod_{j=1}^K (u_1-x_j)^{\beta_j}\right] du_1
\]
where $K=m+n-1$ and $x_j \notin (x_i,x_{i+1})$.

\item The exponents satisfy:
\[
\beta_j \in \left\{-\frac{4}{\kappa}, \frac{8}{\kappa}, \frac{2n-4m+4-\kappa}{\kappa}\right\}, \quad \text{with } \beta_i=\beta_{i+1}=-\frac{4}{\kappa}.
\]
\end{enumerate}

For $\kappa>4$ with $\Gamma_1=[x_i,x_{i+1}]^+$, we make the substitution $u_1(t) = (1-t)x_i + tx_{i+1}$ to obtain:
\[
\mathcal{J}^{(1,K)} \sim \frac{\Gamma(1-4/\kappa)^2}{\Gamma(2-8/\kappa)} (x_{i+1}-x_i)^{1-8/\kappa} \mathcal{N}\left[\prod_{\substack{j=1\\j\neq i,i+1}}^K (x_i-x_j)^{\beta_j}\right].
\]

For $0<\kappa\leq4$ with the Pochhammer contour $\Gamma_1=\mathscr{P}(x_i,x_{i+1})$, the same substitution yields:
\[
\mathcal{J}^{(1,K)} \sim 4e^{-8\pi i/\kappa} \sin^2\left(\frac{4\pi}{\kappa}\right) \frac{\Gamma(1-4/\kappa)^2}{\Gamma(2-8/\kappa)} (x_{i+1}-x_i)^{1-8/\kappa} \mathcal{N}\left[\prod_{\substack{j=1\\j\neq i,i+1}}^K (x_i-x_j)^{\beta_j}\right].
\]

The combined asymptotic behavior takes the universal form:
\[
\lim_{x_i,x_{i+1}\to p} (x_{i+1}-x_i)^{6/\kappa-1} \mathcal{J}^{(m,n)}_\alpha(\boldsymbol{x}) = n(\kappa) \mathcal{J}^{(m,n)}_{\hat{\alpha}}(\boldsymbol{\hat{x}}),
\]
where the prefactor $n(\kappa)$ incorporates:
\begin{itemize}
\item The gamma function combination for $\kappa>4$
\item The phase and sine factors for $\kappa\leq4$
\item The operator product expansion normalization
\end{itemize}

\subsection{Case 3: $x_i$ as contour endpoint, $x_{i+1}$ not} \label{case3}
\begin{figure}[h]
\centering
\includegraphics[width=15cm]{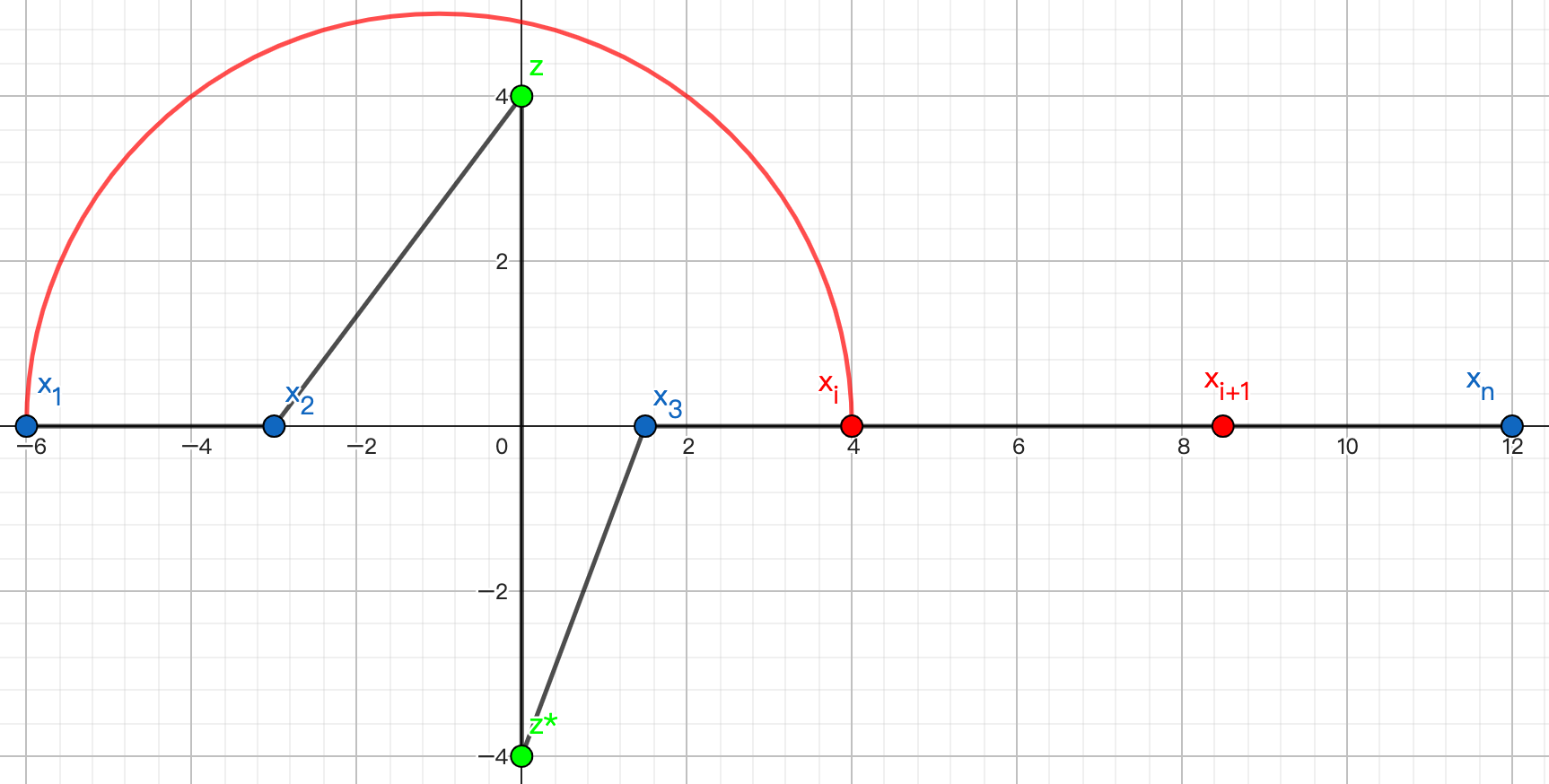}
\caption{Case 3}
\end{figure}
When either $x_i$ or $x_{i+1}$ is an endpoint of a single integration contour $\Gamma_1$ in $\mathcal{J}^{(m,n)}_\alpha(\boldsymbol{x})$, the asymptotic analysis proceeds as follows. Without loss of generality, we consider $x_i$ as the contour endpoint (the case with $x_{i+1}$ as endpoint follows by symmetry). The contour configuration depends on $\kappa$:

\[
\Gamma_1'' = \begin{cases}
\mathscr{P}(x_{i-1},x_i) & \text{for } 0<\kappa\leq4 \\
[x_{i-1},x_i]^+ & \text{for } 4<\kappa<8
\end{cases}
\]

where $[x_{i-1},x_i]^+$ denotes a slight deformation into the upper half-plane and $\mathscr{P}(x_{i-1},x_i)$ is the Pochhammer contour. The relevant integral takes the form:

\begin{equation}
\mathcal{J}^{(1,K)}\left(\{\beta_j\}|\Gamma_1''|x_1,\ldots,x_K\right) = \int_{\Gamma_1''} \mathcal{N}\left[\prod_{j=1}^K (u_1-x_j)^{\beta_j}\right] du_1
\end{equation}

with $K=m+n-1$ and exponents satisfying:

\begin{equation}
\begin{gathered}
s := \sum_{j=1}^K \beta_j = -2, \quad \beta_j \in \{-4/\kappa, 8/\kappa, (2n-4m+4-\kappa)/\kappa\} \\
\beta_{i-1} = \beta_i = \beta_{i+1} = -4/\kappa, \quad \beta_{i+2} \in \{-4/\kappa, (2n-4m+4-\kappa)/\kappa\}
\end{gathered}
\end{equation}

For $\kappa\in(4,8)$, we analyze the definite integrals:

\begin{equation}
I_k(\boldsymbol{x}) := \int_{x_k}^{x_{k+1}} \mathcal{N}\left[\prod_{j=1}^K (u_1-x_j)^{\beta_j}\right] du_1
\end{equation}

Using contour deformation with a large semicircle (radius $R\to\infty$) in the upper/lower half-plane, we obtain the key relation:

\begin{equation}\label{integral identity for Ik}
\sum_{k=1}^K e^{\pm\pi i\sum_{l=1}^k \beta_l} I_k = 0
\end{equation}

\begin{remark}
    Without loss of generality, we may assume that all points $x_1, x_2, \ldots, x_K$ lie on the real line. When the configuration is not collinear, equation~\eqref{integral identity for Ik} still holds up to a multiplicative algebraic factor of the form $e^{2i(\theta_k - \pi)}$ for each term $I_k$, where $\theta_k$ denotes the angle between the two consecutive line segments meeting at $x_k$.
\end{remark}

Solving for $I_{i-1}$ yields:

\begin{equation}
I_{i-1} = -\sum_{k=1}^{i-2} \frac{\sin\pi\sum_{l=k+1}^{i+1}\beta_l}{\sin\pi(\beta_i+\beta_{i+1})} I_k + \sum_{k=i+2}^K \frac{\sin\pi\sum_{l=i+2}^k\beta_l}{\sin\pi(\beta_i+\beta_{i+1})} I_k - \frac{\sin\pi\beta_{i+1}}{\sin\pi(\beta_i+\beta_{i+1})} I_i
\end{equation}

The dominant asymptotic behavior comes from:

\begin{equation}
I_{i-1} \sim -\frac{\sin\pi\beta_{i+1}\Gamma(\beta_i+1)\Gamma(\beta_{i+1}+1)}{\sin\pi(\beta_i+\beta_{i+1})\Gamma(\beta_i+\beta_{i+1}+2)} (x_{i+1}-x_i)^{\beta_i+\beta_{i+1}+1} \mathcal{N}\left[\prod_{j\neq i,i+1}^K (x_i-x_j)^{\beta_j}\right]
\end{equation}

For $\kappa\in(0,4]$, we use the analytically continued form:

\begin{equation}
\int_{x_k}^{x_{k+1}} \mapsto \frac{1}{4e^{\pi i(\beta_k-\beta_{k+1})}\sin\pi\beta_k\sin\pi\beta_{k+1}} \oint_{\mathscr{P}(x_k,x_{k+1})}
\end{equation}

which gives:

\begin{equation}
\mathcal{J}^{(1,K)} \sim 4e^{\pi i(\beta_{i-1}-\beta_i)}\sin\pi\beta_{i-1}\sin\pi\beta_i \cdot \text{(same asymptotic form as above)}
\end{equation}

The final result for both $\kappa$ ranges is:

\begin{equation}
\lim_{x_i,x_{i+1}\to p} (x_{i+1}-x_i)^{6/\kappa-1} \mathcal{J}^{(m,n)}_\alpha(\boldsymbol{x}) = \mathcal{J}^{(m-1,n)}_{\hat{\alpha}}(\boldsymbol{\hat{x}})
\end{equation}

The factor $n(\kappa)^{-1}$ emerges naturally from the product of phase factors and gamma functions in the asymptotic expansion. This result shows how the Coulomb gas integral reduces when one point in the collapsing pair is a contour endpoint while the other is not, with the specific form depending crucially on whether $\kappa$ is above or below 4.

\subsection{Case 4: $x_i$, $x_{i+1}$ are endpoints of two distinct contour} \label{case4}
\begin{figure}[h]
\centering
\includegraphics[width=15cm]{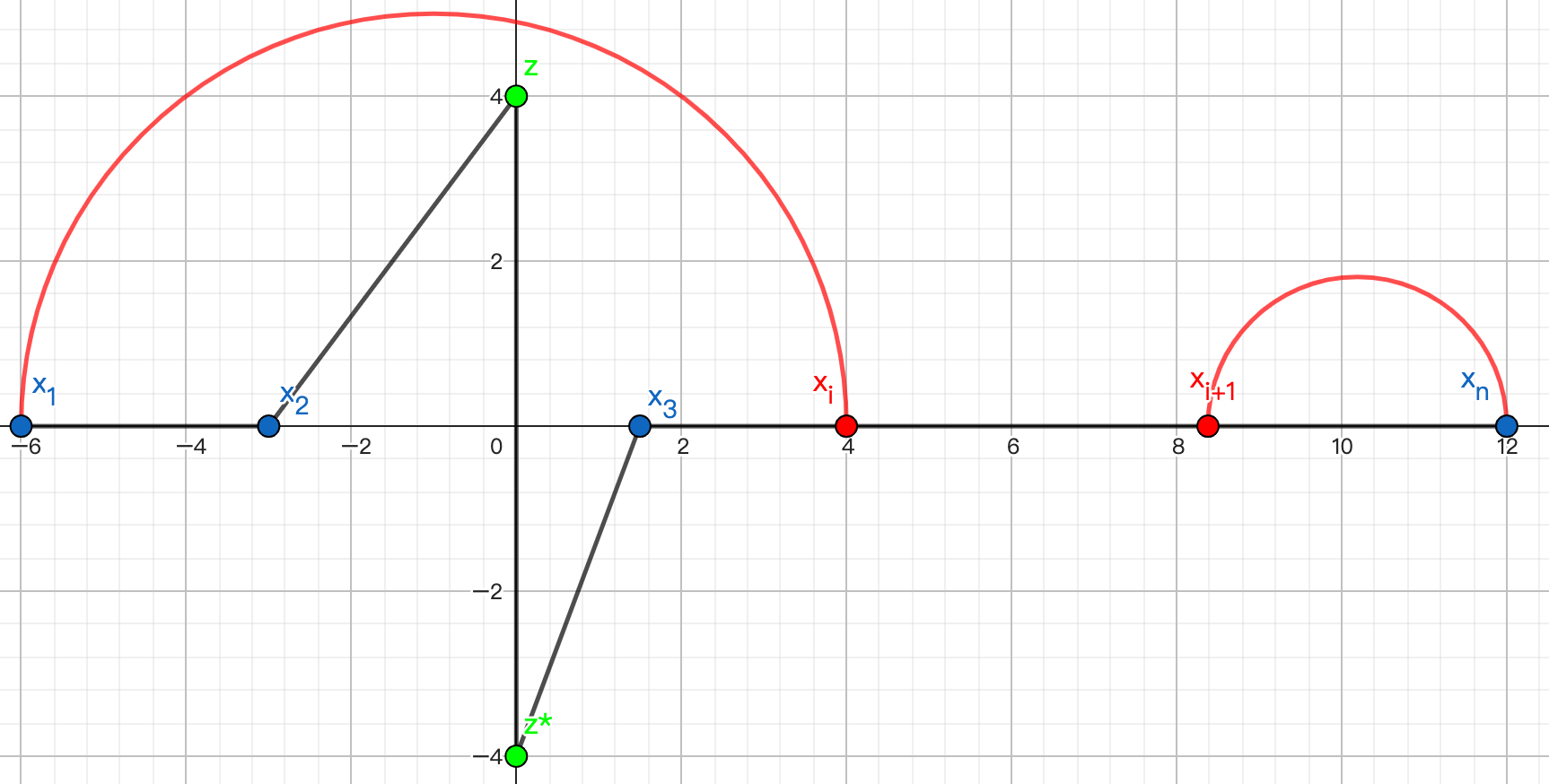}
\caption{Case 4}
\end{figure}
We analyze the asymptotic behavior of the Coulomb gas integral $\mathcal{J}^{(m,n)}(\boldsymbol{x})$ as $x_{i+1}\to x_i$ when $x_i$ is an endpoint of contour $\Gamma_1$ and $x_{i+1}$ is an endpoint of a different contour $\Gamma_2$. For $\kappa>4$ with $\Gamma_1$ and $\Gamma_2$ being simple contours, we decompose:

\[
\Gamma_1 = \begin{cases}
\Gamma_1' \text{ (endpoints at $x_{i-1}$)} \\
\Gamma_1'' \text{ (endpoints at $x_{i-1},x_i$)}
\end{cases}, \quad
\Gamma_2 = \begin{cases}
\Gamma_2' \text{ (endpoints at $x_{i+2}$)} \\
\Gamma_2'' \text{ (endpoints at $x_{i+1},x_{i+2}$)}
\end{cases}
\]

where $\Gamma_1'$ or $\Gamma_2'$ may be empty. The exact form depends on $\kappa$:

\[
\Gamma_1'' = \begin{cases}
\mathscr{P}(x_{i-1},x_i), & 0<\kappa\leq4 \\
[x_{i-1},x_i]^+, & 4<\kappa<8
\end{cases}, \quad
\Gamma_2'' = \begin{cases}
\mathscr{P}(x_{i+1},x_{i+2}), & 0<\kappa\leq4 \\
[x_{i+1},x_{i+2}]^+, & 4<\kappa<8
\end{cases}
\]

Using Fubini's theorem, we first integrate $u_1$ along $\Gamma_1$, then $u_2$ along $\Gamma_2$, followed by $u_3,\ldots,u_m$ along their respective contours. For $l>2$, $(x_{i+1}-u_l)^{-4/\kappa}\to(x_i-u_l)^{-4/\kappa}$ uniformly as $x_{i+1}\to x_i$, so the asymptotic behavior is determined by:

\[
\mathcal{J}^{(2,K)}\left(\{\beta_j\};\gamma|\Gamma_1'',\Gamma_2''|x_1,\ldots,x_K\right) = \iint_{\Gamma_1''\times\Gamma_2''} \mathcal{N}\left[\prod_{j=1}^K (u_1-x_j)^{\beta_j}(u_2-x_j)^{\beta_j}(u_2-u_1)^\gamma\right]du_2 du_1
\]

where $K=m+n-2$ and $x_j\notin(x_{i-1},x_{i+2})$. The exponents satisfy:

\[
\begin{gathered}
s := \sum_{j=1}^K \beta_j + \gamma = -2 \\
\beta_j \in \{-4/\kappa, 8/\kappa, (2n-4m+4-\kappa)/\kappa\} \\
\beta_i = \beta_{i+1} = -4/\kappa, \quad \beta_{i+2} \in \{-4/\kappa, (2n-4m+4-\kappa)/\kappa\}, \quad \gamma = 8/\kappa
\end{gathered}
\]

For $\kappa\in(4,8)$, these imply $s \in \mathbb{Z}^{-}\setminus\{-1\}$, $\beta_i+\beta_{i+1}+\gamma/2 \notin\mathbb{Z}$, $\beta_j > -1$ for all $j$, $\beta_i+\beta_{i+1} < -1$, and $\beta_i+\gamma/2 = 0$.

We define the integrals for $m,k\in\{1,\ldots,K\}$:

\[
I_{m,k}(\boldsymbol{x}) := \int_{x_m}^{x_{m+1}}\int_{x_k}^{x_{k+1}} \mathcal{N}[\cdots] du_2 du_1
\]

with the convention for $k=K$:

\[
\int_{x_K}^{x_{K+1}} := \int_{x_K}^\infty + \int_{-\infty}^{x_1}
\]

For the diagonal terms ($m=k$):

\[
I_{k,k}(\boldsymbol{x}) := \int_{x_{k-1}}^{x_k}\int_{x_{k-1}}^{u_1} \mathcal{N}[\cdots] du_2 du_1 + \int_{x_{k-1}}^{x_k}\int_{u_1}^{x_k} \mathcal{N}[\cdots] du_2 du_1
\]

Using a large semicircular contour (radius $R\to\infty$) in the upper/lower half-plane:

\[ \label{integral identity for Ijk}
\sum_{k=1}^{i-2} e^{\pm\pi i\sum_{l=1}^k \beta_l} I_{i-1,k} + e^{\pm\pi i\sum_{l=1}^{i-1}\beta_l}(1+e^{\pm\pi i\gamma})I_{i-1,i-1} + \sum_{k=i}^K e^{\pm\pi i(\sum_{l=1}^k \beta_l+\gamma)} I_{i-1,k} = 0
\]
\begin{remark}
    Without loss of generality, we may assume that all points $x_1, x_2, \ldots, x_K$ lie on the real line. When the points are not collinear, equation~\eqref{integral identity for Ijk} still holds up to a multiplicative algebraic factor of the form $e^{2i(\theta_j + \theta_k - 2\pi)}$ for each term $I_{j,k}$, where $\theta_j$ and $\theta_k$ denote the angles between two consecutive line segments at $x_j$ and $x_k$, respectively, in the configuration.
\end{remark}

Solving these equations yields:

\begin{equation} 
\begin{aligned}
I_{i-1,i+1} = & \sum_{k=1}^{i-2} \frac{\sin\pi(\sum_{l=k+1}^{i-1}\beta_l+\gamma/2)}{\sin\pi(\beta_i+\beta_{i+1}+\gamma/2)} I_{i-1,k} \\
& - \sum_{k=i+2}^K \frac{\sin\pi(\sum_{l=i}^k\beta_l+\gamma/2)}{\sin\pi(\beta_i+\beta_{i+1}+\gamma/2)} I_{i-1,k} \\
& - \frac{\sin\pi(\beta_i+\gamma/2)}{\sin\pi(\beta_i+\beta_{i+1}+\gamma/2)} I_{i-1,i}
\end{aligned}
\end{equation}

For $k=i$, $I_{i-1,i}$ contains divergent terms ($\beta_i+\beta_{i+1}<-1$) but its coefficient vanishes due to $\beta_i+\gamma/2=0$. For $k\notin\{i,i\pm1\}$:

\[
I_{i-1,k} \sim -\frac{\sin\pi\beta_{i+1}}{\sin\pi(\beta_i+\beta_{i+1})} I_{i,k}
\]

where $I_{i,k}$ falls under Case 2 and diverges as $(x_{i+1}-x_i)^{\beta_i+\beta_{i+1}+1}$.

After careful analysis of all terms, we obtain:

\[
\begin{aligned}
I_{i-1,i+1} &\sim -\frac{\sin\pi\beta_{i+1}\Gamma(\beta_i+1)\Gamma(\beta_{i+1}+1)}{\sin\pi(\beta_i+\beta_{i+1})\Gamma(\beta_i+\beta_{i+1}+2)} (x_{i+1}-x_i)^{\beta_i+\beta_{i+1}+1} \\
&\quad \times \mathcal{N}\left[\prod_{j\neq i,i+1}^K (x_i-x_j)^{\beta_j}\right] \\
&\quad \times \left[\sum_{k=1}^{i-2} \frac{\sin\pi(\sum_{l=k+1}^{i-1}\beta_l+\gamma/2)}{\sin\pi(\beta_i+\beta_{i+1}+\gamma/2)} - \sum_{k=i+2}^K \frac{\sin\pi(\sum_{l=i}^k\beta_l+\gamma/2)}{\sin\pi(\beta_i+\beta_{i+1}+\gamma/2)}\right] \\
&\quad \times \int_{x_k}^{x_{k+1}} \mathcal{N}\left[(u_2-x_i)^{\beta_i+\beta_{i+1}+\gamma} \prod_{j\neq i,i+1}^K (u_2-x_j)^{\beta_j}\right] du_2
\end{aligned}
\]

Introducing the joined contour $\Gamma_0' = [x_{i-1},x_{i+2}]^+$, we obtain the cleaner asymptotic form:

\[
\begin{aligned}
\mathcal{J}^{(2,K)} &\sim \frac{\sin\pi\beta_{i+1}\Gamma(\beta_i+1)\Gamma(\beta_{i+1}+1)}{\sin\pi(\beta_i+\beta_{i+1})\Gamma(\beta_i+\beta_{i+1}+2)} (x_{i+1}-x_i)^{\beta_i+\beta_{i+1}+1} \\
&\quad \times \mathcal{N}\left[\prod_{j\neq i,i+1}^K (x_i-x_j)^{\beta_j}\right] \int_{x_{i-1}}^{x_{i+2}} \mathcal{N}\left[\prod_{j\neq i,i+1}^K (u_2-x_j)^{\beta_j}\right] du_2
\end{aligned}
\]

For $\kappa\leq4$, we use the analytically continued form:

\[
\int_{x_k}^{x_{k+1}} \mapsto \frac{1}{4e^{\pi i(\beta_k-\beta_{k+1})}\sin\pi\beta_k\sin\pi\beta_{k+1}} \oint_{\mathscr{P}(x_k,x_{k+1})}
\]

which yields:

\[
\begin{aligned}
\mathcal{J}^{(2,K)} &\sim 4e^{\pi i(\beta_{i-1}-\beta_i)}\sin\pi\beta_{i-1}\sin\pi\beta_i \\
&\quad \times \frac{\sin\pi\beta_{i+1}\Gamma(\beta_i+1)\Gamma(\beta_{i+1}+1)}{\sin\pi(\beta_i+\beta_{i+1})\Gamma(\beta_i+\beta_{i+1}+2)} (x_{i+1}-x_i)^{\beta_i+\beta_{i+1}+1} \\
&\quad \times \mathcal{N}\left[\prod_{j\neq i,i+1}^K (x_i-x_j)^{\beta_j}\right]
\end{aligned}
\]

Combining all cases, the final asymptotic behavior is:

\[
\lim_{x_i,x_{i+1}\to p} (x_{i+1}-x_i)^{6/\kappa-1} \mathcal{J}^{(m,n)}_\alpha(\boldsymbol{x}) = \mathcal{J}^{(m-1,n-2)}_{\hat{\alpha}}(\boldsymbol{\hat{x}})
\]

\subsection{Case 5: $x_i$, $x_{i+1}$ as endpoints of single contour with complementary arc collapse} \label{case5}
\begin{figure}[h]
\centering
\includegraphics[width=15cm]{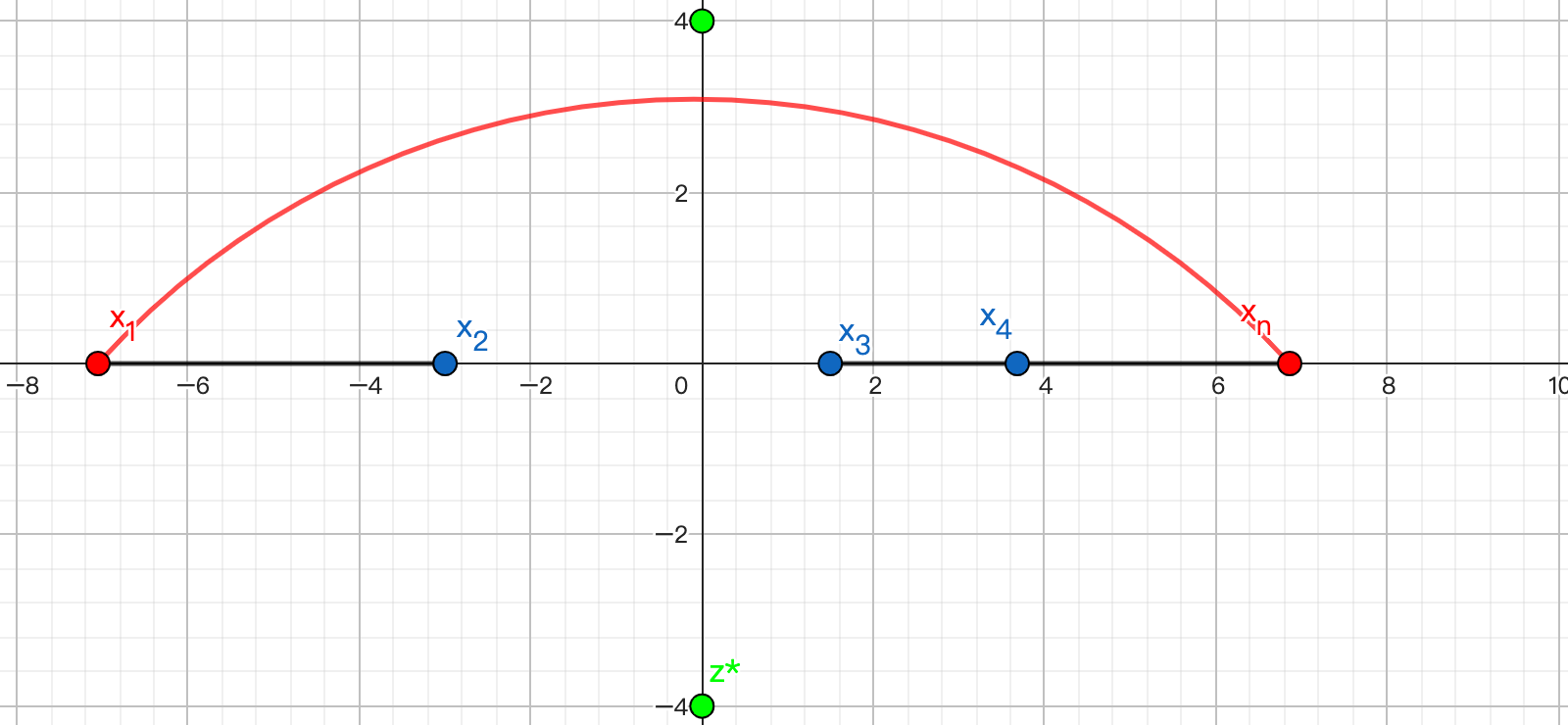}
\caption{Case 5}
\end{figure}

When \( x_i \) and \( x_{i+1} \) are endpoints of the same integration contour \( \Gamma_1 \), but we consider their collapse along the complementary arc. 

Without loss of generality, we may assume that \( x_1 = -R \) and \( x_n = R \) for large \( R \), and study the asymptotic behavior as \( R \to \infty \). In this case, the Coulomb gas integral \( \mathcal{J}^{(m,n)}_\alpha(\boldsymbol{x}) \) admits the following decomposition:
\begin{equation} \label{eq:case5-decomposition}
\begin{split}
\lim_{R\to\infty} (2R)^{6/\kappa - 1} \mathcal{J}^{(m,n)}_\alpha(\boldsymbol{x}) = \;& 
\lim_{R\to\infty} (2R)^{6/\kappa - 1} \left.\mathcal{J}^{(m,n)}_\alpha\right|_{\Gamma_1 \to \Gamma_1'} \\
& + \lim_{R\to\infty} (2R)^{6/\kappa - 1} \left.\mathcal{J}^{(m,n)}_\alpha\right|_{\Gamma_1 \to \Gamma_1''} \\
& + \lim_{R\to\infty} (2R)^{6/\kappa - 1} \left.\mathcal{J}^{(m,n)}_\alpha\right|_{\Gamma_1 \to \Gamma_1'''}.
\end{split}
\end{equation}

With $\kappa>4$ and \( \Gamma_1 \) simple contours, we decompose $\Gamma_1$ as \( \Gamma_1' \): an arc connecting \( (x_1, x_2) \), \( \Gamma_1'' \): the portion of \( \Gamma_1 \) running from \( x_2 \) to \( x_{n-1} \), \( \Gamma_1''' \): the arc connecting \( (x_{n-1}, x_n) \).

\noindent\textbf{Key Observations:}

1. The middle term vanishes:
\begin{equation}
\lim_{R\rightarrow\infty} (2R)^{6/\kappa-1} \left.\mathcal{J}^{(m,n)}_\alpha\right|_{\Gamma_1\to\Gamma_1''} = 0
\end{equation}

2. The remaining terms both reduce to Case 3 configurations, yielding:
\begin{equation}
\lim_{R\rightarrow\infty} (2R)^{6/\kappa-1} \mathcal{J}^{(m,n)}_\alpha(\boldsymbol{x}) = 2 \mathcal{J}^{(m-1,n)}_{\hat{\alpha}}(\boldsymbol{\hat{x}})
\end{equation}

For $\kappa > 4$, the contour behavior is:
\[
\Gamma_1 = [x_i,x_{i+1}]^+ \quad 
\]

For $0 < \kappa \leq 4$, we use the analytically continued version:
\[
\Gamma_1 = \mathscr{P}(x_i,x_{i+1}) \quad 
\]

The factor of 2 arises because:
\begin{itemize}
    \item Both $\Gamma_1'$ and $\Gamma_1'''$ contribute equally
    \item Each produces a $\mathcal{J}^{(m-1,n)}_{\hat{\alpha}}$ term
\end{itemize}

The same decomposition applies in the general setting of collapsing \( x_i \) and \( x_{i+1} \), yielding:
\begin{equation}
\begin{aligned}
\lim _{x_{i}, x_{i+1}\rightarrow p} (x_{i+1} - x_i)^{6/\kappa - 1} \mathcal{J}^{(m,n)}_\alpha(\boldsymbol{x}) = \;& 
\lim_{x_{i+1} \to x_i} (x_{i+1} - x_i)^{6/\kappa - 1} \left.\mathcal{J}^{(m,n)}_\alpha\right|_{\Gamma_1 \mapsto \Gamma_1'} \\
& + \lim_{x_{i+1} \to x_i} (x_{i+1} - x_i)^{6/\kappa - 1} \left.\mathcal{J}^{(m,n)}_\alpha\right|_{\Gamma_1 \mapsto \Gamma_1''} \\
& + \lim_{x_{i+1} \to x_i} (x_{i+1} - x_i)^{6/\kappa - 1} \left.\mathcal{J}^{(m,n)}_\alpha\right|_{\Gamma_1 \mapsto \Gamma_1'''}.
\end{aligned}
\end{equation}

\section{Temperley-Lieb type algebras and meander matrix}

The contents of this chapter are based on the works \cite{RSA14,MDSA13}. For detailed definitions, properties, and proofs related to Temperley--Lieb type algebras, we refer the reader to these references.

\subsection{Temperley--Lieb Algebras and Standard Modules}

Let $q \in \mathbb{C}^\times$ be a nonzero complex parameter and set $\delta := q + q^{-1}$. The \emph{Temperley--Lieb algebra} $\mathrm{TL}_n(\beta)$ is the unital associative algebra over $\mathbb{C}$ generated by $\{1, u_1, u_2, \dots, u_{n-1} \}$, subject to the relations:

\begin{defn}[Temperley--Lieb algebra]
The algebra $\mathrm{TL}_n(\beta)$ satisfies:
\[
u_i^2 = \beta u_i, \quad u_i u_{i\pm 1} u_i = u_i, \quad u_i u_j = u_j u_i \quad \text{for } |i - j| > 1.
\]
\end{defn}

\paragraph{Diagrammatic interpretation.}  
Elements of $\mathrm{TL}_n(\beta)$ admit a basis consisting of planar, non-crossing pairings of $2n$ points arranged on two rows. The multiplication is defined via vertical concatenation, replacing each resulting closed loop with the scalar $\beta$.

\begin{defn}[Link states and module $M_n$]
A \emph{link state} on $n$ nodes is a non-crossing pairing of the $n$ points into $p$ arcs (called links) and $n - 2p$ unpaired points (called defects). The vector space $M_n$ is the span over $\mathbb{C}$ of all such link states. Define the subspaces
\[
M_{n,p} := \text{span of link states with at least } p \text{ arcs},
\]
which yield the descending filtration:
\[
0 \subset M_{n,\lfloor n/2 \rfloor} \subset \cdots \subset M_{n,1} \subset M_{n,0} = M_n.
\]
\end{defn}

\begin{defn}[Standard module]
The standard module $V_{n,p}$ is defined as the graded piece
\[
V_{n,p} := M_{n,p} / M_{n,p+1},
\]
which is a representation of $\mathrm{TL}_n(\beta)$ via diagram concatenation. Its dimension is given by:
\[
\dim V_{n,p} = d_{n,p} = \binom{n}{p} - \binom{n}{p-1},
\]
the $(n,p)$-entry of the Catalan triangle.
\end{defn}

\begin{defn}[Radical and irreducible quotient]
Let $\langle \cdot, \cdot \rangle_{n,p}$ be the invariant bilinear form on $V_{n,p}$ defined below. The \emph{radical} of $V_{n,p}$ is the null space of this form:
\[
R_{n,p} := \{ x \in V_{n,p} \mid \langle x, y \rangle = 0 \text{ for all } y \in V_{n,p} \}.
\]
The \emph{irreducible quotient} of $V_{n,p}$ is
\[
L_{n,p} := V_{n,p} / R_{n,p}.
\]
\end{defn}

\subsection{Inner Product on Standard Modules and Meander Matrix}

\begin{defn}[Inner product on standard module]
Let $x, y \in V_{n,p}$ be two $(n,p)$-link states. Define $\langle x, y \rangle_{n,p}$ as follows:
\begin{enumerate}
    \item Reflect $x$ across a vertical axis (mirror image).
    \item Glue $x$ to $y$ vertically, forming a closed link diagram.
    \item If all defects pair, and $m$ closed loops are formed, then
    \[
    \langle x, y \rangle_{n,p} := \delta^p.
    \]
    \item Otherwise, if any defects remain unpaired, then $\langle x, y \rangle_{n,p} := 0$.
\end{enumerate}
\end{defn}

\begin{thm}[Invariance under TL action]
The bilinear form $\langle \cdot, \cdot \rangle_{n,p}$ is symmetric and $\mathrm{TL}_n$-invariant:
\[
\langle Ux, y \rangle = \langle x, U^\dagger y \rangle, \quad \forall U \in \mathrm{TL}_n,
\]
where the involution $U^\dagger$ reflects the diagram horizontally (reversing composition order).
\end{thm}

\begin{defn}[Gram matrix (Meander matrix)]
Let $\{x_1, \dots, x_{d_{n,p}}\}$ be a basis of $V_{n,p}$. The \emph{Gram matrix} $G_{n,p}$ is defined by:
\[
G_{n,p}(i,j) := \langle x_i, x_j \rangle.
\]
This matrix is symmetric and positive semidefinite. It is also referred to as the \emph{meander matrix}, due to its connection with enumerative problems involving non-crossing paths (meanders).
\end{defn}

\begin{thm}[Radical as Gram kernel]
The radical $R_{n,p}$ is the kernel of the Gram matrix $G_{n,p}$:
\[
R_{n,p} = \ker G_{n,p}.
\]
\end{thm}

\begin{example}[Gram matrices for $n=4$]
Let $n = 4$. The three standard modules $V_{4,p}$ for $p = 0, 1, 2$ have Gram matrices:
\[
G_{4,0} = (1), \quad
G_{4,1} =
\begin{pmatrix}
\beta & 1 & 0 \\
1 & \beta & 1 \\
0 & 1 & \beta
\end{pmatrix}, \quad
G_{4,2} =
\begin{pmatrix}
\delta^2 & \beta \\
\beta & \delta^2
\end{pmatrix}.
\]
\end{example}

\begin{defn}[Irreducibility criterion]
The standard module $V_{n,p}$ is irreducible if and only if the Gram matrix is non-degenerate, i.e.,
\[
\det G_{n,p} \ne 0.
\]
If $\det G_{n,p} = 0$, then $V_{n,p}$ is reducible, and its radical
\[
R_{n,p} := \{x \in V_{n,p} \mid \langle x, y \rangle = 0, \forall y \in V_{n,p}\}
\]
is a nonzero proper submodule.
\end{defn}

\begin{defn}[$q$-integer and criticality]
The $q$-integer is defined for any $m \in \mathbb{Z}_{>0}$ by:
\[
[m]_q := \frac{q^m - q^{-m}}{q - q^{-1}}.
\]
We say that a pair $(n, p)$ is \emph{critical} if
\[
[n - 2p + 1]_q = 0 \quad \text{or equivalently,} \quad q^{2(n - 2p + 1)} = 1.
\]
Critical pairs correspond to values of $q$ for which the Gram determinant may vanish.
\end{defn}

\begin{thm}[Recursive formula for $\det G_{n,p}$]
Let $q \in \mathbb{C}^\times$ and $\beta = q + q^{-1}$. Suppose $(n,p)$ is not critical, i.e., $[n - 2p + 1]_q \ne 0$. Then the Gram determinant satisfies the recurrence:
\[
\det G_{n,p} = \left( [n - 2p + 1]_q \right)^{d_{n-1,p-1}} \cdot \det G_{n-1,p} \cdot \det G_{n-1,p-1},
\]
where $d_{n-1,p-1} = \dim V_{n-1,p-1} = \binom{n-1}{p-1} - \binom{n-1}{p-2}$.
\end{thm}

\begin{defn}[Initial conditions]
The recurrence is initialized by the following:
\begin{itemize}
  \item For $p = 0$, the only basis vector is the fully defective link state, and $G_{n,0} = (1)$, so
  \[
  \det G_{n,0} = 1.
  \]
  \item For $n = 2p$, the standard module $V_{2p,p}$ is one-dimensional (a fully paired link state), and
  \[
  \det G_{2p,p} = \delta^p = (q + q^{-1})^p.
  \]
\end{itemize}
\end{defn}

\begin{example}[Case $p = 1$]
For $p = 1$, the recurrence reduces to
\[
\det G_{n,1} = [n - 1]_q \cdot \det G_{n-1,1},
\]
with $\det G_{2,1} = \delta = q + q^{-1}$. Iterating gives the closed formula:
\[
\det G_{n,1} = \prod_{j=1}^{n-2} [j]_q.
\]
This is the $q$-analog of $(n-2)!$.
\end{example}

\subsection{Affine Temperley--Lieb Algebra and Standard Modules}

We define the affine Temperley--Lieb algebra $\mathrm{aTL}_n(\delta)$, describe its diagrammatic realization, and introduce a family of standard modules equipped with a natural inner product. The determinant of the corresponding Gram matrix, known as the affine meander matrix, is also computed.

\begin{defn}[Affine Temperley--Lieb algebra]
Fix $n \geq 2$ and $q \in \mathbb{C}^\times$, and set $\delta := -q - q^{-1}$. The affine Temperley--Lieb algebra $\mathrm{aTL}_n := \mathrm{aTL}_n(\delta)$ is the unital associative algebra generated by $\{\Omega_n, \Omega_n^{-1}, \mathrm{id}, e_0, e_1, \dots, e_{n-1}\}$, subject to the relations (with indices modulo $n$):
\[
\begin{aligned}
& e_i^2 = \delta e_i, \quad e_i e_{i \pm 1} e_i = e_i, \quad e_i e_j = e_j e_i \quad \text{if } |i - j| \geq 2, \\
& \Omega_n e_i = e_{i-1} \Omega_n, \quad (\Omega_n^{\pm 1} e_0)^{n-1} = \Omega_n^{\pm n} (\Omega_n^{\pm 1} e_0), \quad \Omega_n \Omega_n^{-1} = \Omega_n^{-1} \Omega_n = \mathrm{id}.
\end{aligned}
\]
These relations imply that $e_i = \Omega_n^{1 - i} e_1 \Omega_n^{i - 1}$ for all $i$, so the algebra is generated by $\{e_1, \Omega_n, \Omega_n^{-1}\}$.
\end{defn}

\begin{defn}[Affine link diagram]
An $(m, n)$-diagram with $2m \leq n$ is a rectangle with $n-2m$ marked points on the left and $n$ on the right, with the top and bottom edges identified to form a cylinder. Each point is connected to another by a non-crossing arc, possibly wrapping around the cylinder. Diagrams are considered equivalent under isotopy.
\end{defn}

\begin{defn}[Affine Standard Module]
Fix an even integer $n \in \mathbb{Z}_{\geq 0}$. The \emph{affine standard module} $\widetilde{V}_{n,m}$ is the complex vector space spanned by affine link diagrams with $n$ marked boundary points and $m = \frac{n - d}{2}$ non-crossing arcs (named links), where $d$ is the number of through-lines (vertical strands connecting left to right, named defects). The basis elements are equivalence classes of affine link diagrams up to isotopy on the cylinder, and are denoted by $\widetilde{B}_{n}^{d}$.
\end{defn}

\begin{remark}
The module $\widetilde{V}_{n,m}$ carries a natural action of the affine Temperley--Lieb algebra $\mathrm{aTL}_n(\beta)$, defined diagrammatically by concatenation followed by loop removal and normalization. This module is indecomposable but typically not simple, and its structure encodes important representation-theoretic and combinatorial information.
\end{remark}

\subsection{Inner Product and the Affine Meander Matrix}

\begin{defn}[Affine standard module and Gram product]
Let $\widetilde{V}_{n,m}$ be the standard module, define the Gram inner product $\langle \cdot \mid \cdot \rangle_G$ on $\widetilde{V}_{n,m}$ by:
\[
\left\langle \alpha \mid \beta \right\rangle_G := 
\begin{cases}
0 & \text{if any two defects from } \alpha \text{ or } \beta \text{ are connected}, \\
a^{n_a} b^{n_b} & \text{otherwise},
\end{cases}
\]
where $n_a$ and $n_b$ denote the numbers of contractible and non-contractible loops formed in the concatenated diagram $\mathcal{G}(\alpha, \beta)$, and $a, b$ are loop weights.
\end{defn}

\begin{thm}[Determinant of affine meander matrix]
Let $\widetilde{\mathcal{G}}_n^d$ denote the Gram matrix for the standard module with $n$ nodes and $d = n - 2m$ through-lines. Then:
\[
\begin{aligned}
\det \widetilde{\mathcal{G}}_n^0 &= \prod_{k=1}^{n/2} \left(a^2 - 4 \cos^2\left( \frac{4k\pi}{\kappa} \right) \right) \binom{n}{n/2 - k}, \\
\det \widetilde{\mathcal{G}}_n^d &= \prod_{k=1}^{m} \left(4 - 4 \cos^2\left( \frac{4(k + d/2)\pi}{\kappa} \right) \right) \binom{n}{(n-d)/2 - k}, \quad \text{for } d > 0, \\
\det \widetilde{\mathcal{G}}_n^n &= 1.
\end{aligned}
\]
\end{thm}

\begin{proof}[Sketch of proof]
The result follows from decomposing the standard module into angular momentum sectors, each corresponding to a Fourier mode of the rotation operator on the cylinder. The Gram matrix becomes block-diagonal, and each block contributes an eigenvalue involving $C_k = \cos\left( \frac{4k\pi}{\kappa} \right)$, leading to the product formula.
\end{proof}

\begin{remark}
    The determinant of the affine meander matrix yields non-zero for $\kappa$ irrational.
\end{remark}

\section{Asymptotics of partition functions and connectivity}
\label{asymptotics and connectivity}

\subsection{Multiple chordal SLE($\kappa$) connectivity}
The following two sections are summarized from \cite{KP16}.
We construct the partition functions \( \mathcal{Z}_\alpha \) corresponding to extremal multiple SLEs with deterministic connectivities encoded by link patterns \( \alpha \in \mathrm{LP}(n,m) \). Each \( \mathcal{Z}_\alpha \) satisfies the same system of partial differential equations and conformal covariance, but their boundary conditions---formulated as asymptotic constraints---depend explicitly on \( \alpha \).

In the chordal setup, the geometry dictates whether curves meet pairwise as prescribed by \( \alpha \). In local multiple SLEs, curve intersections are not directly meaningful; however, the possibility that \( |X_t - X_t^{(i)}| \to 0 \) is encoded in the behavior of the drift term \( b^{(j)} = \kappa \frac{\partial \mathcal{Z}}{\partial x_j} \), which depends on \( \mathcal{Z} \).

The asymptotic behavior near collisions \( x_j \to x_{j+1} \) is governed by the following limits:

\begin{equation} \label{fusion1}
\lim_{x_j, x_{j+1} \to \xi} \frac{\mathcal{Z}(x_1, \ldots, x_n)}{(x_{j+1} - x_j)^{(\kappa - 6)/\kappa}} \end{equation}

If this limit vanishes, the refined limit

\begin{equation} \label{fusion2}
\lim_{x_j, x_{j+1} \to \xi} \frac{\mathcal{Z}(x_1, \ldots, x_n)}{(x_{j+1} - x_j)^{2/\kappa}} \end{equation}

exists and is non-zero.

\begin{thm}
Fix \( j \in \{1, \ldots, n\} \), and suppose \( \mathcal{Z} \) is a positive solution to the system (1.2) such that
\[ \left|x_j - x_{j\pm 1}\right|^{-\Delta} \mathcal{Z}(x_1, \ldots, x_n) \to C \neq 0 \]
as \( x_j, x_{j\pm 1} \to \xi \), uniformly for fixed \( x_k \) with \( k \neq j, j\pm 1 \). Then the process \( Y_t = |X_t - X_t^{(j\pm 1)}| \) is absolutely continuous with respect to a time-changed Bessel process of dimension
\[ \delta = 1 + 2\Delta + \frac{4}{\kappa}. \]
Depending on the asymptotics, two cases arise:
\begin{itemize}
  \item If the limit \eqref{fusion1} is non-zero, then \( \Delta = \frac{\kappa - 6}{\kappa} \) and \( \delta = 3 - \frac{8}{\kappa} < 2 \), so the curves can collide (\( Y_\tau = 0 \)) with positive probability.
  \item If the limit \eqref{fusion1} vanishes but \eqref{fusion2} is non-zero, then \( \Delta = \frac{2}{\kappa} \), so \( \delta = 1 + \frac{8}{\kappa} > 2 \), and \( Y_t > 0 \) almost surely for \( t \leq \tau \).
\end{itemize}
\end{thm}

\begin{proof}[Sketch of proof]
We construct three measures \( \mathbb{P}, \mathbb{P}_\varnothing, \overline{\mathbb{P}} \) on Loewner evolutions where all dynamics reduce to the driving function \( X_t \), with other points transported via the Loewner flow. Girsanov's theorem is used to compare these measures via Radon--Nikodym derivatives:
\[
M_t = \prod_{i \neq j} (g_t'(X_0^{(i)}))^{(6-\kappa)/2\kappa} \cdot \mathcal{Z}(X_t^{(1)}, \ldots, X_t^{(n)}),
\]
\[
\bar{M}_t = (g_t'(X_0^{(j\pm1)}))^{(6-\kappa)/2\kappa} \cdot |X_t - X_t^{(j\pm1)}|^\Delta.
\]

Their ratio is bounded under the asymptotic assumptions, ensuring mutual absolute continuity. The resulting SDE for \( Y_t = |X_t - X_t^{(j\pm1)}| \) under \( \overline{\mathbb{P}} \) takes the form:
\[
\mathrm{d} Y_t = \sqrt{\kappa} \, \mathrm{d}\widetilde{B}_t + \left(1 + \frac{\kappa \Delta}{2}\right) \cot\left(\frac{Y_t}{2}\right) \mathrm{d}t,
\]
which matches the form of a Bessel-type process with effective dimension \( \delta \).
\end{proof}

\begin{defn}[Pure Partition Functions] \label{pure partition function}
The functions $\mathcal{Z}_\alpha: \mathfrak{X}_n \rightarrow \mathbb{R}^+$, indexed by link patterns $\alpha \in \mathrm{LP}(n,m)$, are called \emph{pure partition functions}. They are a collection of positive solutions to the null vector equations, subject to boundary conditions specified by their asymptotic behavior, which is determined by the link pattern $\alpha$
\begin{itemize}
    \item[\textnormal{(ASY)}] \textbf{Asymptotics:} For all $\alpha \in \mathrm{LP}_(n,m)$, $j \in \{1, \ldots, n\}$, and $\xi \in (x_{j-1}, x_{j+2})$, the following limit exists:
    \[
    \lim_{x_j, x_{j+1} \rightarrow \xi} \frac{\mathcal{Z}_\alpha\left(x_1, \ldots, x_n\right)}{(x_{j+1} - x_j)^{\frac{6 - \kappa}{\kappa}}}
    =
    \begin{cases}
        0, & \text{if } \{j, j+1\} \notin \alpha, \\
        \mathcal{Z}_{\hat{\alpha}}\left(x_1, \ldots, x_{j-1}, x_{j+2}, \ldots, x_n\right), & \text{if } \{j, j+1\} \in \alpha,
    \end{cases}
    \]
\end{itemize}
where $\hat{\alpha} = \alpha / \{j, j+1\} \in \mathrm{LP}(n-2,m-1)$ denotes the link pattern obtained from $\alpha$ by removing the link $\{j, j+1\}$ and relabeling the remaining indices as $1, 2, \ldots, n-2$.
\end{defn}

\begin{conjecture}[Pure partition function- Coulomb gas integral]
For irrational $\kappa \in(0,8)$, the pure partition functions are related to Coulomb gas integrals by affine meander matrix:
\begin{equation}
\mathcal{J}_\beta^{(m,n)}(\boldsymbol{x})=\sum_{\alpha \in \operatorname{LP}(n,m)} \mathcal{M}_{\kappa}(\alpha, \beta) \mathcal{Z}_\alpha(\boldsymbol{x}), \quad  \beta \in \mathrm{LP}(n,m).
\end{equation}

Conversely, we have 

\begin{equation}
\mathcal{Z}_\beta(\boldsymbol{x})=\sum_{\alpha \in \operatorname{LP}(n,m)} \mathcal{M}_\kappa(\alpha, \beta)^{-1} \mathcal{J}_\alpha^{(m,n)}(\boldsymbol{x}), \quad  \beta \in \mathrm{LP}(n,m).
\end{equation}
\end{conjecture}

Although the positivity of the pure partition functions $\mathcal{Z}_\beta$ is not rigorously established, asymptotic analysis of the Coulomb gas ground state solutions confirms that $\mathcal{Z}_\beta$ satisfies the required boundary asymptotics.

\subsection{Multiple radial SLE($\kappa$) connectivity}
In this section, we aim to construct the partition functions $\mathcal{Z}_\alpha$ associated with multiple radial SLE($\kappa$) systems, where the deterministic connectivity is encoded by a link pattern $\alpha \in \mathrm{LP}_N$. Each $\mathcal{Z}_\alpha$ must satisfy a shared system of partial differential equations and conformal covariance conditions. However, the asymptotic behavior of $\mathcal{Z}_\alpha$ near colliding marked points reflects the connectivity pattern $\alpha$, and thus provides the necessary boundary data for selecting the desired solution.

In the case of multiple SLEs with pure geometry, we wish the random curves to connect marked points in pairs according to $\alpha$. While local multiple SLEs do not produce literal intersections, the connectivity type determines whether or not certain pairs of driving processes can approach each other. Specifically, for the system of processes $(X_t^{(1)}, \ldots, X_t^{(n)})$, the connectivity pattern dictates whether the difference $|X_t - X_t^{(j\pm 1)}|$ can hit zero in finite time. Since the drift term $b^{(j)}$ in the SDE for $X_t^{(j)}$ depends on $\mathcal{Z}$, the behavior of this difference process is encoded in the short-distance asymptotics of the partition function.

We examine the fusion behavior of $\mathcal{Z}$ as two marked points coalesce. The following limit always exists:
\begin{equation} 
\lim_{x_j, x_{j+1} \to \xi} \frac{\mathcal{Z}(x_1,\ldots,x_n)}{(x_{j+1}-x_j)^{\frac{\kappa-6}{\kappa}}}.
\end{equation}
If the limit vanishes, then a refined fusion limit exists:
\begin{equation} 
\lim_{x_j, x_{j+1} \to \xi} \frac{\mathcal{Z}(x_1,\ldots,x_n)}{(x_{j+1}-x_j)^{\frac{2}{\kappa}}}.
\end{equation}

These asymptotics determine whether collisions of driving processes occur, and we formalize this through the following theorem.

\begin{thm}
Let $j \in \{1,\ldots,n\}$ and suppose $\mathcal{Z}$ is a positive solution to the system of PDEs, such that the limit
\[
\lim_{x_j,x_{j\pm1} \to \xi} \left|x_j - x_{j\pm1}\right|^{-\Delta} \cdot \mathcal{Z}(x_1,\ldots,x_n)
\]
exists and is nonzero for all fixed $\{x_k\}_{k \neq j, j\pm1}$ and $\xi$ in a suitable interval. Let $(X_t^{(1)},\ldots,X_t^{(n)})$ solve the SDE with drift $b^{(j)} = \kappa \, \partial_j \log \mathcal{Z}$. Let $\tau$ be any stopping time such that $|X_t^{(k)} - X_t^{(l)}| \geq \varepsilon$ and $|X_t^{(k)} - X_t^{(l)}| \leq 2\pi - \varepsilon$ for all $k \neq l$ and $t \leq \tau$. Then the process $Y_t = |X_t - X_t^{(j\pm1)}|$ is absolutely continuous with respect to a linear time change of a Bessel process of dimension
\[
\delta = 1 + 2\Delta + \frac{4}{\kappa}.
\]

Depending on the value of the fusion limit:
\begin{itemize}
    \item If the limit \eqref{fusion1} is nonzero, then $\Delta = \frac{\kappa - 6}{\kappa}$ and $\delta = 3 - \frac{8}{\kappa} < 2$, implying that $Y_t$ may hit zero with positive probability.
    \item If the limit \eqref{fusion1} vanishes but the limit \eqref{fusion2} is nonzero, then $\Delta = \frac{2}{\kappa}$ and $\delta = 1 + \frac{8}{\kappa} > 2$, so $Y_t$ remains strictly positive almost surely.
\end{itemize}
\end{thm}

\begin{proof}
We focus on the case with sign $-$; the other case is analogous. Let $\mathrm{P}$ denote the law of the SDE, and let $\mathrm{P}_{\emptyset}$ be the law of a standard radial $\mathrm{SLE}_\kappa$, with driving process $X_t = \sqrt{\kappa} B_t + X_0$. Define the Radon--Nikodym derivative:
\[
M_t = \prod_{i \neq j} \left(g_t'(X_0^{(i)})\right)^{\frac{6-\kappa}{2\kappa}} \cdot \mathcal{Z}(X_t^{(1)}, \ldots, X_t^{(n)}).
\]
Then $\mathrm{P}$ is absolutely continuous with respect to $\mathrm{P}_{\emptyset}$ up to any localizing sequence $\tau_n \uparrow \tau$.

Next, define an auxiliary measure $\overline{\mathrm{P}}$ via
\[
\bar{M}_t = \left(g_t'(X_0^{(j-1)})\right)^{\frac{6-\kappa}{2\pi}} \cdot (X_t - X_t^{(j-1)})^\Delta.
\]
Under the assumption on $\mathcal{Z}$, the ratio $M_t / \bar{M}_t$ is bounded away from $0$ and $\infty$ for all $t \leq \tau$, so $\mathrm{P}$ and $\overline{\mathrm{P}}$ are mutually absolutely continuous. Under $\overline{\mathrm{P}}$, the process $X_t$ satisfies
\[
dX_t = \sqrt{\kappa} \, d\tilde{B}_t + \frac{\kappa\Delta}{2} \cot\left( \frac{2}{X_t - X_t^{(j-1)}} \right) dt.
\]
Therefore, $Y_t = X_t - X_t^{(j-1)}$ satisfies
\[
dY_t = \sqrt{\kappa} \, d\tilde{B}_t + \left(1 + \frac{\kappa\Delta}{2}\right) \cot\left( \frac{2}{Y_t} \right) dt.
\]
This corresponds to a linear time change of a Bessel process of dimension $\delta = 1 + 2\Delta + \frac{4}{\kappa}$. The final conclusions follow from known properties of Bessel processes.
\end{proof}

\begin{defn}[Pure Partition Functions] 
The functions $\mathcal{Z}_\alpha: \mathfrak{X}_n \rightarrow \mathbb{R}^+$, indexed by link patterns $\alpha \in \mathrm{LP}(n,m)$, are called \emph{pure partition functions}. They are a collection of positive solutions to the null vector equation, subject to boundary conditions specified by their asymptotic behavior, which is determined by the link pattern $\alpha$
\begin{itemize}
    \item[\textnormal{(ASY)}] \textbf{Asymptotics:} For all $\alpha \in \mathrm{LP}_n$, $j \in \{1, \ldots, n\}$, and $\xi \in (\theta_{j-1}, \theta_{j+2})$, the following limit exists:
    \[
    \lim_{\theta_j, \theta_{j+1} \rightarrow \xi} \frac{\mathcal{Z}_\alpha\left(\theta_1, \ldots, \theta_n\right)}{(\theta_{j+1} - \theta_j)^{\frac{6 - \kappa}{\kappa}}}
    =
    \begin{cases}
        0, & \text{if } \{j, j+1\} \notin \alpha, \\
        \mathcal{Z}_{\hat{\alpha}}\left(\theta_1, \ldots, \theta_{j-1}, \theta_{j+2}, \ldots, \theta_n\right), & \text{if } \{j, j+1\} \in \alpha,
    \end{cases}
    \]
\end{itemize}
where $\hat{\alpha} = \alpha / \{j, j+1\} \in \mathrm{LP}(n-2,m-1)$ denotes the link pattern obtained from $\alpha$ by removing the link $\{j, j+1\}$ and relabeling the remaining indices as $1, 2, \ldots, n-2$.
\end{defn}

\begin{conjecture}[Pure partition function- Coulomb gas integral]
For irrational $\kappa \in(0,8)$, the pure partition functions are related to Coulomb gas integrals by affine meander matrix:
\begin{equation}
\mathcal{J}_\beta^{(m,n)}(\boldsymbol{\theta})=\sum_{\alpha \in \operatorname{LP}(n,m)} \mathcal{M}_{\kappa}(\alpha, \beta) \mathcal{Z}_\alpha(\boldsymbol{\theta}), \quad  \beta \in \mathrm{LP}(n,m).
\end{equation}

Conversely, we have 

\begin{equation}
\mathcal{Z}_\beta(\boldsymbol{\theta})=\sum_{\alpha \in \operatorname{LP}(n,m)} \mathcal{M}_\kappa(\alpha, \beta)^{-1} \mathcal{J}_\alpha^{(m,n)}(\boldsymbol{\theta}), \quad  \beta \in \mathrm{LP}(n,m).
\end{equation}
\end{conjecture}

Although the positivity of the pure partition functions $\mathcal{Z}_\beta$ is not rigorously established, asymptotic analysis of the Coulomb gas ground state solutions confirms that $\mathcal{Z}_\beta$ satisfies the required boundary asymptotics.
\section{Asymptotics of Excited Solutions}
\label{asymptotic of exicited solutions}

The methodology developed for ground state solutions naturally extends to the analysis of excited solutions in both radial and chordal multiple SLE($\kappa$) systems. 

In the Coulomb gas formalism, such solutions are represented by integrals that include an additional integration variable~$\zeta_1$, corresponding to the excitation. This variable is integrated over an outer contour surrounding the marked point, yielding a $(m+1)$-fold integral of the form
\begin{equation}
    \mathcal{K}^{(m,n)}_\alpha(\boldsymbol{x}) := \oint_{C(\zeta_1)} d\zeta_1 \oint_{\mathcal{C}_1} \cdots \oint_{\mathcal{C}_m} \Phi_\kappa(\boldsymbol{x}, \zeta_1, \boldsymbol{\xi}) \, d\xi_m \cdots d\xi_1,
\end{equation}
where $\boldsymbol{\xi} = (\xi_1, \ldots, \xi_m)$ are the screening charges, and $\Phi_\kappa$ is the Coulomb gas integrand depending on all variables.

To facilitate the asymptotic analysis under the limit where pairs of points in~$\boldsymbol{x}$ coalesce, we exchange the order of integration. That is, we treat~$\zeta_1$ as a fixed parameter and first evaluate the $m$-fold Coulomb gas integral:
\begin{equation}
    F(\zeta_1; \boldsymbol{x}) := \oint_{\mathcal{C}_1} \cdots \oint_{\mathcal{C}_m} \Phi_\kappa(\boldsymbol{x}, \zeta_1, \boldsymbol{\xi}) \, d\xi_m \cdots d\xi_1.
\end{equation}
We subsequently coalesce pairs of boundary points and finally reduce to the contour integral:
\begin{equation}
     \oint_{C(\zeta_1)} F(\zeta_1) \, d\zeta_1.
\end{equation}

This strategy isolates the contribution of the excitation and makes the structure of the excited solution amenable to precise asymptotic analysis.

\printbibliography

\end{document}